\documentclass[draft, 10pt]{amsart}%
\usepackage[applemac]{inputenc}
\usepackage{amsfonts}
\usepackage{latexsym, amssymb}
\usepackage{amsmath}
\usepackage{amssymb}
\usepackage{graphicx}%
\numberwithin{equation}{section}
\newcommand{\field}[1]{\mathbb{#1}}
\newcommand{\R}{\field{R}}

\newcommand{\N}{\field{N}}

\def\D\theta ij#1{\dis \frac{\partial #1}{\partial \theta_i^j}}

\def\sqr{{\hskip1pt\vcenter{\vbox{\hrule height.4pt
\hbox{\vrule width.4pt height4pt\kern4pt
\vrule width.4pt}
\hrule height.4pt}}}}

\def\cal{\mathcal}

\def\trans{\ifmmode{\rm \frown\mkern-16.8mu \vert
\mkern8mu}\else{$\frown\mkern-16.8mu\vert\mkern8mu$}\fi\relax}
\def\ap{\rightarrow}

\def\dis{\displaystyle}
\def\D{\Delta}
\def\dis{\displaystyle}
\def\a{\alpha}

\def\g{\gamma}
\def\G{\Gamma}
\def\t{\tau}
\def\d{\delta}

\def\l{\lambda}
\def\L{\hat{\epsilon}}
\def\n{\nu}

\def\s{\sigma}
\def\S{\Sigma}

\def\p{\partial}

\newtheorem{The1}{Theorem}
\newtheorem{The}{Theorem}[subsection]
\newtheorem{Pro}{Proposition}[subsection]
\newtheorem{Def}{Definition}[subsection]
\newtheorem{Cor}{Corollary}[section]
\newtheorem{Lem}{Lemma}[subsection]
\newtheorem{Ex}{Example}[subsection]
\newtheorem{Exs}{Examples}[subsection]
\newtheorem{Rem}{Remark}[subsection]

\newtheorem{Obs}{Observations}[subsection]

\begin{document}
\title{On Finsler  entropy of smooth distributions and Stefan-Sussman foliations  }
\author{ F. Pelletier}
\date{}
\maketitle

\begin{abstract}
Using the definition of   entropy of  a family of increasing distances  on a compact metric set given in \cite{DNS} we introduce a notion of Finsler entropy for smooth distributions and Stefan-Sussmann  foliations. This concept generalizes most of classical  topological entropy on a compact Riemannian manifold :
the entropy of a flow (\cite{Di}),  of a regular foliation (\cite{GLW}), of  a regular distribution (\cite{Bis}) and  of a geometrical structure (\cite{Zu}). The essential results of this paper is   the nullity of  the Finsler  entropy  for a controllable  distribution and  for a singular Riemannian foliation.
\end{abstract}
\section{Introduction and Results}
A notion of {\it geometric entropy for regular foliations}   in  compact Riemannian manifolds was introduced by Ghys, Langevin, and Walczak (\cite{GLW}). The basic idea is to try to measure  the transversal complexity of the the  leaves.  These authors shows in particular that when this geometric entropy vanishes, the foliation admits a transverse measure. The reader can find in \cite{Hu} a survey of the relation between the nullity or not nullity of this entropy and some geometrical properties  of the foliations. 
A notion of {\it entropy of a regular distribution} was proposed by Bi\'s in \cite{Bis} by comparing the distance of sets of curves tangent to the distribution and which start from two  points of a fixed "transversal" to this distribution. In particular, he proves that if the distribution is integrable we recover the previous geometrical entropy of the foliation. More recently,  Zung uses  an analog concept in \cite{Zu} to define the {\it  entropy of a   geometrical structure}  that is the data of a vector bundle $A\ap M$ on $M$, a morphism $\sharp: A\ap TM$ and a collection of norms on each fiber of $A$. \\
Now   in the context  of laminations by hyperbolic Riemann surfaces,  Dinh, Nguyen and  Sibony introduce in \cite{DNS} a notion of topological entropy for a family of increasing distances  on a compact metric set. It is precisely this last approach that  we use  to define the {\it Finsler entropy} for smooth distributions and Stefan-Sussmann foliations. More precisely as in the framework of control theory we consider a {\it set of admissible  curves} in a compact metrics space $(X,d)$. Such  a set ${\cal A}$ is a set  of continuous  curves which is stable by  reparametrizations, concatenations and restrictions and contains constant curves. Then to any  filtration $\{{\cal A}_r\}_{r\in \R^+}$ of $\cal A$ which is increasing with $r$, we can associate  of  set $\{d_r\}_{r\in \R^+}$ of distances on $X$. In fact such a   distance $d_r$ can be seen  as the Hausdorff distance of  sets curves in ${\cal A}_r$ which start from a  given point in $X$ and which are parametrized on $[0,1]$ (see Remark \ref{Hausdorff}). In this way  $\{d_r\}_{r\in \R^+}$ is a set of  increasing  distances with respect to $r$ (for more details see section \ref{adcurves}). If we choose the set of absolutely  continuous curves tangent to a smooth distribution  {\it a.e.}, we get a set of admissible curves in the previous sense. Now given a Finsler metric on $D$ ({\it cf.} Definition \ref{Finslerdistribution}), we can define the length of  curves tangent to $D$  {\it a.e.}. Then the subset    of admissible curves of length at most $r$ gives rise to a filtration and we can associate a family of distances as we have already seen. Finally we define the {Finsler entropy } of $D$ as the topological entropy corresponding to this family of distances on a compact Finsler manifold $(M,\Phi)$ considered as metric space for the distance $d_\Phi$ associated to $\Phi$. If $\cal F$ is a Stefan-Sussmann foliation, its Finsler entropy is the Finsler entropy of the  distribution defined by $\cal F$. The Finsler entropy is a natural generalization of the entropy of a regular distribution defined by Bi\'s and also the entropy of a geometric structure defined by Zung.\\
  In this paper we give some basic properties of the Finsler entropy of a smooth distribution. On the other hand, according to the famous result of accessibility of Sussmann \cite{Su}, to a smooth distribution $D$ is associated a  canonical Stefan-Sussmann foliation whose leaves are the accessibility sets for $D$ that is  the set of points which can be joined by tangent curves to $D$. In fact the Finsler entropy depends of  this foliation. In particular we have
  
  \begin{The1}\label{The1}${}$\\
 Let $D$ be a smooth distribution on a compact Finsler  manifold  $(M,\Phi)$ and any Finsler metric $F$ on $D$. If the distribution $D$ is controllable ({\it i.e.} any two points of $M$ can be joined by a curve tangent to $D$) then the Finsler entropy is zero.
 \end{The1}
In the continuation of the famous works of Molino  on Riemann foliations (see for instance \cite{Mo})  an important activity of research  is concerned by  {\it singular Riemannian foliations} (see \cite{ABT} and references inside this paper). It is well known that the geometrical entropy of a regular Riemannian foliation is zero. In our context we also have :

\begin{The1}\label{The2}${}$\\
Let $\cal F$ be a singular Riemannian foliation on a compact Riemannian manifold $(M,g)$. The Finsler entropy of $\cal F$  (relative to the induced Riemannian structure on each leaf)  is zero.\\
\end{The1}

This paper is organized as follow. In Section 2  we begin by recalling the context of the entropy of a family of  increasing distances exposed in \cite{DNS} and then we develop the notion of entropy for an admissible set of continuous curves in a compact metric set. We end this section by classical examples of entropy which are particular cases of the entropy of a family of increasing distances. The definitions  and results about Finsler entropy of  smooth distributions and Stefan-Sussmann foliations are contained in Section 3 and in particular the proof of Theorem \ref{The1}. For more concise and complete results about this topic the reader can consult Observations \ref{propentrop}. The last section essentially concerns the proof of Theorem \ref{The2}. We begin this section by some results about   Finsler entropy relative to isometric Finsler submersion and
to smooth maps which are isometric Finsler submersion between distributions. After recalling essential results about singular Riemannian  foliations we end by a proof of Theorem \ref{The2}.

\section{  On the entropy of a family of distances}
\subsection{ Entropy of a family of distances}\label{entropydr}${}$\\
We recall the general concept of entropy introduced in \cite{DNS}.\\
Let $(X,d)$ a compact metric space. Given a set $\Lambda=\N$ or $\Lambda=\R^+$, consider a family ${\cal M}=\{d_\l\}_{\l\in \L}$  of distances on $X$ such that $d_0=d$ and $d_\l$ is increasing with respect to $\l$. We will say that ${\cal M}=\{d_\l\}_{\l\in \Lambda}$ is {\bf an increasing distances on $X$}.
Given any $\epsilon >0$ we denote by $M(d_\l,\epsilon)$ the minimum number of balls of radius $\epsilon$ with respect to $d_\l$ needed to cover $X$. 

\begin{Def}\label{entropy}${}$\\
The entropy of $X$ with respect to $\cal M$ is 
$h({\cal M},X)=\dis\sup_{\epsilon>0}\limsup_{\l\ap\infty}\frac{1}{\l}\ln M(d_\l,\epsilon)$
\end{Def}

Since $d_\l$ is increasing with $\l$, $M(d_\l,\epsilon)$ is  increasing with respect to $\l$ and $\limsup_{\l\ap\infty}\frac{1}{\l}\ln M(d_\l,\epsilon)$ is increasing   when $\epsilon$ decreases. Therefore we have 
$$h({\cal M},X)=\dis\lim_{\epsilon\ap 0^+}\limsup_{\l\ap\infty}\frac{1}{\l}\ln M(d_\l,\epsilon)$$

We can also defined $h({\cal M},X)$ in the following way: a subset $A$ of $X$ is called $(d_\l,\epsilon)$-separated if any pair $(x,y)$ of distinct points of $A$ we have $d_\l(x,y)>\epsilon$. Let $N(d_\l,\epsilon)$ the maximal cardinal of a $(d_\l,\epsilon)$-separated set in $X$. Then we have the relation ({\it cf.} Proposition 3.1 \cite{DNS})
$$M(d_\l,\epsilon)\leq N(d_\l,\epsilon)\leq M(d_\l,\epsilon/2)$$
and we obtain
\begin{eqnarray}\label{entropyN}
h({\cal M},X)=\dis\lim_{\epsilon\ap 0^+}\limsup_{\l\ap\infty}\frac{1}{\l}\ln N(d_\l,\epsilon)
\end{eqnarray}

\begin{Rem}\label{compactX}${}$
\begin{enumerate}
\item In all this section, we have assume that $(X,d)$ is compact metric space. However more generally we can consider any metric space $(X,d)$ and any family  ${\cal M}=\{d_\l\}_{\l\in \Lambda}$  of distances on $X$ such that $d_0=d$ and $d_\l$ is increasing with respect to $\l$  then for any relatively compact subset $K$ of $X$ we get a family ${\cal M}_K$ of induced increasing  distances on $K$ and as previously way the entropy $h({\cal M}_K,K)$ can be  well defined. Note that from Definition \ref{entropy}, it follows that  if $\overline{K}$ is the closure of $K$  we have $h({\cal M}_K,K)=h({\cal M}_{\overline{K}},\overline{K})$. Now if 
 $K$ and $K'$ are two relatively compact subset of $X$  and then  we have ({\it cf.} \cite{AnSk})
$$h({\cal M}_{K\cup K'},K\cup K')=\sup\{h({\cal M}_K,K),h({\cal M}_{K'},K')\}.$$
In particular if $K\subset K'$ then $ h({\cal M}_K,K)\leq h({\cal M}_{K'},K')$.
\item Assume that we have another family ${\cal M}'=\{d'_\l\}_{\l\in \Lambda}$ of increasing distances  on $X$
 such that \\$d'_\l\geq C d_\l$ for some $C>0$ and for any $\l\in\Lambda$ then we have ({\it cf.} \cite{DNS})
 $$h({\cal M}',X)\geq h({\cal M},X)$$
 \end{enumerate}
  \end{Rem}
  
  
\begin{Pro}\label{hproduct}${}$\\
Let  ${\cal M}'=\{d'_\l\}_{\l\in \Lambda}$ (resp.  ${\cal M}''=\{d''_\l\}_{\l\in \Lambda}$)  of increasing distances on $X'$ (resp. $X''$)  such that $(X',d'_0)$ (resp. $(X'',d''_0)$) is a compact metric space.
\begin{enumerate}
\item  On $X=X'\times X''$ we consider the family of distance ${\cal M}=\{d_\l\}_{\l\in\Lambda}$ with $d_\l=\max\{d'_\l,d''_\l\}$. Then we have
$$h({\cal M},X)=h({\cal M}',X')+h({\cal M}'',X'')$$
\item let $f: X'\ap X''$ a surjective map  which is $K$-Lipschitzian from $(X',d'_\l)$ to $(X'', d''_\l)$ for all $\l$. Then we have $$ h({\cal M}',X')\geq h({\cal M}'',X'').$$
\end{enumerate}
\end{Pro}
\begin{proof}${}$\\ If we have a covering of $X'$  by $M'_\l$ balls of radius $\epsilon$ relative to $d'_\l$ and a covering of $X''$  by $M''_\l$ balls of radius $\epsilon$ relative to $d''_\l$ then by cartesian products of balls we get a covering of $X$ by $M'_\l\times M''_\l$ balls of radius $\epsilon$ of $X$. Therefore
$$M(d_\l,\epsilon)\leq  M(d'_\l,\epsilon)\times M(d''_\l,\epsilon)$$
 if $U'$ is a ball of radius $\epsilon/2$  in $X'$ relative to $d'_r$, then we need at least $M(d''_\l,\epsilon)$ balls of radius $\epsilon$ relative to $d''_\l$ to obtain a covering of $U'\times X''$ by the cartesian product of $U'$ with each one of these balls. Therefore we have
 $$M(d_\l,\epsilon)\geq M(d'_\l,\epsilon/2)\times M(d''_\l,\epsilon)$$
 The result of Point (1) is then a direct consequence of  the properties of $\ln$ and these two inequalities.\\
 For the proof of Point (2), we denote by $B'_\l(x',\epsilon)$ and  $B''_\l(x'',\epsilon)$  the open ball of center $x'$ and $x''$ of radius $\epsilon$ in $X'$ and $X''$ relatively to $d'_\l$ and $d''_\l$ respectively. From the assumption, we have
  $d''_\l(f(x'),f(y'))\leq K d'(x',y')$ and so
 $f(B'_\l (x',\epsilon))\subset B''_\l(x'',K\epsilon)$
 for any $x'$ in $f^{-1}(x'')$. Since  $f$ is surjective, we must have 
 $$M(d'_\l,\epsilon)\geq M(d''_\l,K\epsilon)$$
 which implies the announced result.\\
 \end{proof}

 We end by  recalling the following sufficient conditions  of finiteness of the entropy given in  \cite{DNS}

\begin{Pro}\label{hfinite}${}$\\
Let $(X,d)$ a compact metric space such that  that there exists  positive constants $A$  and $m$ such that for any $\epsilon >0$ small enough $X$ admits a covering by balls of radius $\epsilon$ of cardinal bounded by  $A(\epsilon)^{-m}$. Assume that  a family ${\cal M}=\{d_\l\}_{\l\in \Lambda}$ of increasing distances satisfies the following properties:

(1)  there exists  positive constants $A$  and $m$ such that for any $\epsilon >0$ small enough $X$ admits a covering by balls of radius $\epsilon$ of cardinal bounded by  $A(\epsilon)^{-m}$;

(2) $d_\l\leq e^{a\l+b} d$ for some constants $a,b\geq 0$.

\noindent Then,
the entropy $h({\cal M},X)$ is  bounded by ${m.a}$.
\end{Pro}
\subsection{Entropy and admissible curves}\label{adcurves}${}$\\
Given a compact metric space $(X,d)$,
 let denote by ${\cal C}^0(M)$ the set of  continuous curves $\g:[a,b]\ap M$ where $[a,b]$ is any closed interval in $\R$.
 
 \begin{Def}\label{admissible}${}$\\
 a subset $\cal A$ of ${\cal C}^0(M)$ is called a set of admissible  curves if we have the following properties
 \begin{enumerate}
 \item[(i)] ${\cal A}$ contains any constant curve;
  \item[(ii)] if a curve $\g:[a,b]\ap M$ belongs to $\cal A$  then for any  monotonic  continuous map $\t:[c,d]\ap [a,b]$ 
   the curve $\g\circ \t:[c,d]\ap M$ also belongs to $\cal A$;
  \item[(iii)] if two curves $\g:[a,b]\ap M$ and $\g':[a',b']\ap M$ belongs to $\cal A$ and satisfies $\g(b)=\g'(a')$ then the concatenation $\g\star \g':[a, b+(b'-a')]\ap M$ defined by

  $(\g\star \g')(t)=\g(t) $ for $t\in[a,b]$  and $(\g\star \g')(b+s)=\g'(a'+s)$ for $s\in [0,b'-a']$
  
  belongs to $\cal A$;
    \item[(iv)] if a curve $\g:[a,b]\ap M$ belongs to $\cal A$, for any sub-interval $I\subset[a,b]$, the restriction of $\g$ to $I$ belongs to $\cal A$.
  \end{enumerate}
  \end{Def}
   Note that if $\g$ is an admissible curves of $\cal A$, there always exists  a new parametrization $\t:[0,1]\ap [a,b]$ such that $\hat{\g}=\g\circ \t$ is defined on $[0,1]$ and so $\hat{\g}$
    is admissible. Therefore without loss of generality we may  assume that {\bf an admissible curve $\g\in {\cal A}$ is defined on $[0,1]$.}
   
   \begin{Def}\label{filtration}${}$\\
   Given a set of admissible curves $\cal A$, a filtration of $\cal A$ is a family of subsets ${\cal A}_{r\in \R^+}$ of $\cal A$ such that 
    \begin{enumerate}
   \item[(i)] ${\cal A}_0$ is the set of constant curves.
    \item[(ii)] For $0<s< r$   then  ${\cal A}_0\subset{\cal A}_s\subset{\cal A}_r$ and
    if $\g\in{\cal A}_r$ is defined on $[0,1]$ then there exists a sub-interval $[0,T]\subset[0,1]$ and a curve $\mu$ in ${\cal A}_s$ defined on $[0,1]$ such that $\mu([0,1])=\g([0,T])$
    \item[(iii)] ${\cal A}=\dis\bigcup_{r\in \R^+}{\cal A}_r$
 \end{enumerate}
 \end{Def}
 Fix some filtration ${\cal A}_{r\in \R^+}$ of a set $\cal A$ of admissible curves. For any $x\in M$ we put 
 $${\cal A}_r(x)=\{\g\in {\cal A}_r: \g : [0,1]\ap M, \;\g(0)=x\}.$$ Now given
 two points $x$ and $y$ of $M$ we set:
 \begin{eqnarray}\label{deltar}
 \d_r(x,y)=\dis\sup_{\g\in{\cal A}_r(x)}\inf_{\mu\in {\cal A}_r(y)}\sup_{t\in [0,1]} d(\g(t),\mu(t))
\end{eqnarray}

\begin{eqnarray}\label{dr}
d_r(x,y)=\d_r(x,y)+\d_r(y,x)
    \end{eqnarray}
\begin{Rem} \label{interpretationdelta}${}$\\
Denote by ${\cal C}^0([0,1], M)$ the set of continuous curve defined on $[0,1]$ in $M$. Then 
$$\bar{d}(\g,\g')=\dis\sup_{t\in [0,1]}d(\g(t),\g'(t))$$
 is a distance on ${\cal C}^0([0,1], M)$  and since $M$ is compact,  $({\cal C}^0([0,1], M),d)$ is a complet metric space. Now $\dis\inf_{\mu\in {\cal A}_r(y)}\sup_{t\in [0,1]} d(\g(t),\mu(t))$ is nothing but else $\bar{d}(\g,{\cal A}_r(y))$ and finally  $\d_r(x,y)=\dis\sup_{\g\in {\cal A}_r(x)} \bar{d}(\g,{\cal A}_r(y))$. \\
   Therefore , for any integer $n>0$ there exists $g_n\in ,{\cal A}_r(x)$ such that 
 \begin{eqnarray}\label{approximinf}
\d_r(x,y)-1/n\leq \bar{d}(\g_n,{\cal A}_r(y))\leq \d_r(x,y)
\end{eqnarray}
 and for any integer $p>0$ there exists $\mu_p\in {\cal A}_r(y)$ such that  
 $$\bar{d}(\g_n,{\cal A}_r(y))\leq \bar{d}(\g_n,\mu_p)\leq \bar{d}(\g_n,{\cal A}_r(y))+1/p$$

It follows that   for any integers $n>0$, there exist $\g_n\in {\cal A}_r(x)$ and $\mu_n\in{\cal A}_r(y)$ such that
\begin{eqnarray}\label{approxim}
\d_r(x,y)-1/n\leq \bar{d}(\g_n,\mu_n)\leq \d_r(x,y)+1/n.
\end{eqnarray}
\end{Rem}

\begin{Rem}\label{Hausdorff}${}$\\
Since $X$ is compact, each set ${\cal A}_r(x)$ is bounded and we denote by $\overline{{\cal A}_r(x)}$ the closure of ${\cal A}_r(x)$ in the metric space $(({\cal C}^0([0,1],X), \bar{d})$
and  we have  
$$\d_r(x,y)=\dis\sup_{\g\in {\cal A}_r(x)} \bar{d}(\g,{\cal A}_r(y))=\dis\sup_{\g\in \overline{\cal A}_r(x)} \bar{d}(\g,\overline{{\cal A}_r(y)})$$
 The Hausdorff distance $\d^H$  between $\overline{{\cal A}_r(x)}$ and $\overline{{\cal A}_r(y)}$ is well defined  and so is given by
 $$ \d^H(\overline{{\cal A}_r(x)} ,\overline{{\cal A}_r(y)})=\max\{\d_r(x,y), \d_r(y,x)\}$$
  Therefore we get 
  $$ \d^H(\overline{{\cal A}_r(x)} ,\overline{{\cal A}_r(y)})\leq d_r(x,y)\leq 2  \d^H(\overline{{\cal A}_r(x)} ,\overline{{\cal A}_r(y)}).$$\\
 \end{Rem}
  \bigskip
Now the family  $\{d_r\}$ has the following properties:\\
\begin{Pro}\label{popdr}${}$\\
The map $d_r:M\times M\ap \R$ is a distance on $M$ and for all $r\geq s>0$ we have $d_r\geq d_s\geq 2d$
\end{Pro}

\bigskip
\begin{proof} (compare with the proof of Proposition 2.1 in \cite{Zu})
 ${}$\\
We adopt the notations of Remark  \ref{interpretationdelta}. At first note that for any $\g\in {\cal A}_r(x)$ and $\mu\in {\cal A}_r(y)$ we have  $\bar{d}(\g,\mu)\geq d(\g(0),\mu(0))=d(x,y)$. Therefore $\d_r(x,y)\geq d(x,y)$ and so $d_r\geq 2d$. It follows that $d_r(x,y)=0$ if and only if $x=y$. By construction $d_r$ is symmetric. For the triangular inequality choose three points $x$, $y$, $z$ in $M$. 
According to (\ref{approximinf}) for any $n>0$ there exists $\g_n\in {\cal A}_r(x)$ and $\mu_n\in {\cal A}_r(y)$ such that 
$$\d_r(x,y)-1/n\leq \bar{d}(\g_n,\mu_n)\leq \d_r(x,y).$$
 Now consider any $\nu\in{\cal A}_r(z)$. Since $\bar{d}$ is a distance, we have
$$\d_r(x,y)-1/n\leq \bar{d}(\g_n,\mu_n)\leq \bar{d}(\g_n,\nu)+\bar{d}(\nu,\mu_n)$$
for all 
 $\nu\in{\cal A}_r(z)$. Therefore for any integer $n>0$ we get:
$$\d_r(x,y)-1/n\leq \bar{d}(\g_n,{\cal A}_r(z))+\bar{d}(\mu_n,{\cal A}_r(z))\leq \d_r(x,z)+\d_r(y,z).$$
This implies the triangular inequality for $d_r$.\\
It remains to show that $d_r\geq d_s$ when  $r\geq s>0$. It is sufficient to prove  $\d_r\geq \d_s$ when  $r\geq s>0$
 According to (\ref{approximinf}), for any integer $n>0$ let $\g_n\in{\cal A}_s(x)$
 such that 
$$\d_s(x,y)-1/n\leq \bar{d}(\g_n,{\cal A}_s(y))\leq\d_s(x,y).$$
As in Remark \ref{interpretationdelta},  for any integer  $p>0$, there exists $\mu_p\in {\cal A}_r(y)$ such that
\begin{eqnarray}
\bar{d}(\g_n,{\cal A}_r(y))\leq \bar{d}(\g_n,\mu_p)\leq \bar{d}(\g_n,{\cal A}_r(y))+1/p.
\end{eqnarray}
From Property (ii) of Definition \ref{filtration}, there exists a sub-interval $[0,c_p]\subset[0,1]$ such that $\hat{\mu}_p(t)=\mu_p(tc_p)$ belongs to ${\cal A}_s(y)$.  Therefore we have
\begin{eqnarray}\label{1dr}
\bar{d}(\g_n,\hat{\mu}_p)\leq \bar{d}(\g_n,{\mu}_p)\leq \bar{d}(\g_n,{\cal A}_r(y))+1/p\leq \d_r(x,y)+1/p.
\end{eqnarray}
But clearly we have
\begin{eqnarray}\label{2ds}
\d_s(x,y)-1/n\leq \bar{d}(\g_n,{\cal A}_s(y))\leq \bar{d}(\g_n,\hat{\mu}_p).
\end{eqnarray}
Finally from (\ref{1dr}) and (\ref{2ds}) we obtain
$$\d_s(x,y)-1/n\leq\d_r(x,y)+1/p$$
for any integer $n>0$ and $p>0$ which ends the proof.\\
\end{proof}




According to Proposition \ref{popdr}, the family of distance $\{d_r\}_{r\in \R^+}$ associated to a filtration $\{{\cal A}_r\}_{r\in \R^+}$ of a set $\cal A$ of admissible curves  is increasing with $r$. Therefore from Definition \ref{entropy} we have:

\begin{Def}\label{entropy A}${}$\\
The entropy  $h({\cal A},X,d)$ of a set $\cal A$ of admissible curves of $X$ provided with  the family of distances $\{d_r\}$ associated to a filtration ${\cal A}_r$ is the entropy $h(\{d_r\},X)$.
\end{Def}

\subsection{Examples of entropy which is defined by a family of distances}${}$\\
We now present classical situations of entropy which can be defined  as in Definition \ref{entropy} or Definition \ref{entropy A}.

\subsubsection{Topological entropy of a continuous map}\label{topmap}${}$\\
Let  $f:X\ap X$ be a continuous map on a compact metric space $(X,d)$. For $n\in \N$, set $f^0: =Id$ and for $n>0$ $f^n:= f\circ f^{n-1}$ and denote by 
$$d_n(x,y)=\dis\sup_{0\leq i\leq n}d(f^i(x), f^i(y)$$
Clearly, ${\cal M}=\{d_n\}_{n\in \N}$ is an increasing family of distances on $X$ and therefore we can define the entropy $h({\cal M},X)$. This is exactly the canonical topological entropy of $f$ defined  for instance in \cite{Bo}.
Assume that  $(M,g)$ is a  compact  Riemannian manifold,  and $f$ is $k$-Lipchitzian. Then for the Riemannian distance $d_g$ the assumption (1) of Proposition \ref{hfinite} is satisfied ( {\it cf.} \cite{Wa}) and the assumption (2) is satisfied  for   $a=\ln k$ and $b=0$. Therefore topological entropy of $f$ is finite.

\subsubsection{Topological entropy of a pseudo group}${}$\\
Let $\G$ a pseudo group of Local homeomorphisms of a compact metric space $(X,d)$. Assume that $\G$ is generated by a finite set $\G_0$ of $\G$ such that $\G_0$ contains the identity and if $g$ belongs to $\G_0$ then $g^{-1}$ also belongs to $\G_0$. We denote by $\G_n$  the set of all well defined  composition $g_1\circ\cdots\circ g_i$ for $0\leq i\leq n$ of elements $g_1,\dots,g_i$ in $\G_0$. We can consider  distance:
$$d_n(x,y)=\dis\sup_{g\in \G_n} d(g(x),d(g(y))$$
We  obtain a family of increasing   distances ${\cal M}=\{d_n\}_{n\in \N}$ to which is associated the entropy $h({\cal M}, X)$. Then we get the topological entropy of the pair $(\G,\G_0)$ as defined by Candel-Conlon \cite{CaCo1}- \cite{CaCo2} and Walczak \cite{Wa}. Again if $(M,g)$ is a compact Riemannian  manifold and each element of $\G_0$ is Lipschtzian  by same arguments as in subsection \ref{topmap} the entropy of the pair $(\G,\G_0)$ is finite  (see also \cite{Wa} and \cite{GLW}).

\subsubsection{Topological entropy of a vector field}${}$\\
Consider a compact Riemannian metric $(M,g)$ and $Z$ a $C^1$ vector field on $M$. The flow $\phi_t$ of $Z$ is then complete and we can consider the family of metric
$$d_r(x,y)=\dis\sup_{0\leq s\leq r}d(\phi_t(x),\phi_t(y)$$
Again, ${\cal M}=\{d_r\}_{r\in \R^+}$ is an increasing family of distances on $M$. Therefore we can consider the  corresponding entropy $h({\cal M},X)$. This is exactly the topological entropy of $Z$  (see for example  \cite{Di}). Again in this case the entropy is finite (same arguments as  in subsection \ref{topmap} or \cite{Wa} and \cite{GLW}

\subsubsection{Entropy of a regular distribution}\label{entropyregD}${}$\\
A regular distribution $D$ on a compact manifold $M$ is a subbundle of $TM$. We fix  a Riemannian metric $g$ on $M$ and denote by $d$ the associated distance. Consider the set $\cal A$ of absolutely continuous curve $\g:[a,b]\ap M$ which are tangent {\it a.e.} to $D$. We denote by $l(\g)$ the length of a curve $\g$ relative to $g$.
We can consider the filtration $\{{\cal A}_r\}_{r\in\R^+}$ defined by
$${\cal A}_r=\{\g \in {\cal A} : l(\g)\leq r\}$$
Clearly this filtration satisfies the assumption of Definition \ref{filtration}. On the one hand, according to Proposition \ref{popdr}, to this filtration is  associated a family  ${\cal M}=\{d_r\}_{r\in \R^+}$ of increasing distances. Therefore we can defined the entropy $h({\cal A},M,d)$ according to Definition \ref{entropy A}.\\
On the other hand,
 following \cite{Bis}, recall that a {\it complete transversal} is a submanifold $T$ of $M$ of dimension $q=$codim $D$ such that for any $x\in M$, there exists $\g\in {\cal A}$  which joins $x$ to $T$. If $N(d_r,\epsilon, T)$ is the maximal cardinal of any subset of $T$ which is $(d_r,\epsilon)$-separated ({\it cf.} section \ref{entropydr}), then we define
$$h(D,T)=\dis\lim_{\epsilon\ap 0^+}\limsup_{r\ap\infty}\frac{1}{r}\ln N(d_r,\epsilon,T)$$
In fact  the entropy $h(D,T)$ is nothing but else the previous defined entropy number $h({\cal A},M,d)$.
In particular  $h(D,T)$ is independent of the choice of such a transversal $T$.\\ Indeed, it is clear that $N(d_r,\epsilon, T)\leq N(d_r,\epsilon)$. Now if $\a_x$ is the minimal length of a curve $\g\in {\cal A}$ which joins $x$ to $T$ we set 
$$\a=\dis\sup_{x\in M}\a_x$$
Then clearly we have
$$N(d_r,\epsilon)\leq N(d_{r+\a},\epsilon, T)$$
it follows that by passing to the limit when $r\ap\infty$ we get the announced result.

\subsubsection{Geometrical entropy of a regular foliation}\label{folentrop}${}$\\
 Let ${\cal F}$ be a regular foliation on a compact manifold $M$ Ghys, Langevin and Walczak have defined a geometrical entropy $h_{GLW}({\cal F})$ of $\cal F$ relative  to a Riemannian metric on $M$ (see \cite{GLW}). By Theorem B of \cite{Bis}, the entropy $h_{GLW}({\cal F})$ is equal to  $h(D,T)$ defined in subsection  \ref{entropyregD} where $D$
 is the distribution tangent to $\cal F$ and $T$ is any complete transversal. Aigain in this case $h_{GLW}(F)$ is finite (see \cite{GLW} and  \cite{Wa})
\subsubsection{Geometrical entropy of an anchored bundle} (\cite{Zu})\label{anchoredentrypy}${}$\\
An anchored bundle $(A,M,\sharp)$ on a smooth connected manifold $M$ is  the data of a  vector bundle $p:A\ap M$ and  an anchor which is a bundle map $\sharp : A\ap TM $. Then $D=\sharp(A)$ is a distribution on $M$ which is smooth in the sense of \cite {Su} (see also section \ref{smoothD}).
 We denote by $A_x$ the fiber $p^{-1}(x)$ for any $x\in M$ and  by ${D}_x=\sharp(A_x)$. We provide  each fiber $A_x$  with a norm $||\;||$. \\
 An absolutely  curve $\g:[a,b]\ap M$ is called {\bf $A$-admissible} if there exists a measurable section $u:[a,b]\ap A$ over $\g$ such that $\sharp(\g,u)=\dot{\g}$ for almost   $t\in [a,b]$. The path $\hat{\g}$ will be called a {\bf $A$-path} and we denote by ${\cal A}$ the set of $A$-path. We provide  each fiber $A$  with a norm $||\;||$ on each fiber $A_x$\footnote{In \cite{Zu}, the data $(A,M,\sharp,||\;||)$ is called a geometrical structure on $M$}. We have then a natural filtration $\{{\cal A}_r\}_{r\in \R^+}$ on  $\cal A$ defined by:
 $${\cal A}_r=\{\g\in {\cal A} \textrm{ such that } \exists u:[a,b]\ap A \textrm{ over  } \g,\;\sharp(\g,u)=\dot{\g},\;  ||{u}(t)||\leq r\; \forall t\in [a,b] \;{\it a.e.}\}$$
Clearly ${\cal A}_0$ is the set of constant curves in $\cal A$ and we have ${\cal A}_0\subset {\cal A}_s\subset{\cal A}_r$ for $0<s<r$. Now, let $\g\in {\cal A}_r$. There exists a section  $u: [0,1]\ap A$  over$\g$ such that  $\sharp(\g,u)=\dot{\g}$ and   $||u(t)||\leq r$ for all $ t\in [0,1]$ {\it a.e.}. We put
$\hat{\g}(t)=(\g(t),u(t))$ 
If we set $\t=\frac{s}{r}t$, consider the curve $c(\t)=\g(\frac{s}{r}t)$ for $ t\in [0,1]$. Then we have 
\begin{center}
$\frac{dc}{d\t}=\sharp(c(\t),\frac{s}{r}u(\t)).$
\end{center}
Thus we obtain a curve $\hat{c}:[0,1]\ap A$ defined by
\begin{center}
$\hat{c}(\t)=(c(\t),v(\t))=\frac{s}{r}u(\t) \textrm{ for } \t\in [0,\frac{s}{r}] \textrm{ and } \hat{c}(\t)=(c(\frac{s}{r}),v(\frac{s}{r}))$ for all $\t\in[\frac{s}{r},1]$.
\end{center}
and  then $p\circ \hat{c}$ belongs to ${\cal A}_s$ (as  announced in \cite{Zu}).\\

It follows that the previous filtration $\{{\cal A}_r\}_{r\in \R^+}$ satisfies the assumption of  definition \ref{filtration}. Now  if we provide $M$ with a Riemannian metric $g$ and $d$ is the associated distance, by Proposition \ref{popdr} the associated family ${\cal M}=\{d_r\}_{r\in\R^+}$ of increasing permits to define the entropy
$h({\cal A},M,d)$ which is exactly the {\it entropy of the geometric structure} $h(A,M,\sharp,||\;||)$ defined in \cite{Zu}.\\
When $A$ is an  subbundle of $TM$ and  if the norm  $||\;||$ on $A$ is the Riemannian induced norm, the geometric entropy  $h(A,M,\sharp,||\;||)$ is then the entropy $h(A,T)$ as defined in \cite{Bis} ({\it cf.} Remark 3.5).  Moreover, if $\cal A$ is integrable,  we have   $h(A,M,\sharp,||\;||)=h(A)=h_{GLW}({\cal F})$ where $\cal F$ is   the foliation defined  by $A$ ({\it cf.}  Theorem 3.9 in \cite{Zu}).\\
Recall that this entropy can be zero (for instance if $\sharp$ is surjective) or strictly positive (see \cite{Zu} for Examples of such situations ).\\

{\it A particular case of  norm on an anchored bundle is the context of Finsler metric.}\\

\begin{Def}\label{Minkowski}${}$
\begin{enumerate}
\item A Minkowski norm on a vector space $E$ is norm $F:V\ap \R^+$ which is smooth on $E\setminus \{0\}$ and such that for any $u\in E\setminus\{ 0\}$ the quadratic form $g_{u}(v,w):=\dis\frac{1}{2}\frac{\p^2 {\cal F}^2}{\p s \p t}(x,u+sv+tw)_{| s,t=0}$ is definite positive for all $v,w\in E$.
\item A Finsler metric $F$ on $A$  is a  smooth map $F:A\ap \R^+$ such that $F(x,\;)$ is a Minkowski norm on each fiber $A_x$
\end{enumerate}
\end{Def}

In this context, the filtration $\{{\cal A}_r\}_{r\in \R^+}$ on  $\cal A$ defined by:
$${\cal A}_r=\{\g\in {\cal A}\;:\; \exists u:[a,b]\ap A \textrm{ over  } \g,\; \dot{\g}(t)=\sharp(\g(t),u(t)),\;  F(\g(t),{u}(t))\leq r\; \forall t\in [a,b] \;{\it a.e.}\}$$

We choose any Finsler metric $\Phi$ on $M$ and we denote by $d_\Phi$ the associated distance. As previously we obtain an entropy $h(A,M,\sharp,F,\Phi)$ which will be called the
{\bf Finsler entropy} of $(A,M,\sharp)$.

\subsubsection{Entropy and admissible curves for a Stefan-Sussmann  foliations}\label{entropiefol}${}$\\
According to \cite{St} and \cite{Ku}  we have:

\begin{Def}\label{steffol}${}$\\
A Stefan-Sussmann foliation  on a smooth manifold $M$ is   a partition $\cal F$ of $M$ into connected immersed submanifolds called leaves  which fulfills the following property:

for each $x \in  M$, there exists a local chart $(D_\phi,\phi)$ on $M$ around $x$ with
the following properties :

(a) $\phi$ is a surjection $D_\phi\ap  U_\phi\times W_\phi$ where $U_\phi$, $W_\phi$ are open neighbourhoods of $0$ in $\R^k$ and $\R^{n-k}$ respectively, and $k$ is the dimension of the leaf through $x$; 

(b) $\phi(x)=(0,0)$;

(c) If $L$ belongs to $\cal F$, then $\phi(L\cap D_\phi)=U_\phi\times l_{\phi,L}$ where
$$l_{\phi,L}= \{ w\in W_\phi\;:\; \phi^{-1}(0,w)\in L\}.$$

A chart $(D_\phi,\phi)$ which satisfies  the above condition is called a {\it distinguished chart} around $x$.
\end{Def}

Given a Stefan-Sussmann foliation $\cal F$ on a compact manifold $M$, consider a  set $\cal A$ of admissible absolutely continuous curves $\g:[a,b]\ap M$. We say that $\cal A$ is compatible with $\cal F$ if   the following property is satisfied:
\begin{eqnarray}\label{adcompatible}
\forall \g\in {\cal A} \textrm{ defined on } [a,b],\;\;  \g([a,b])\textrm{ is contained in some leaf }  L.
\end{eqnarray}


 We consider a Stefan-Sussmann  foliation $\cal F$ on a compact manifold $M$ and $d$ a distance on $M$  which defines the topology of manifold of $M$. Let $\cal A$ be  a set of admissible curves compatible with $\cal F$. Given  a filtration ${\cal A}_{r\in \R^+}$ which fulfills the assumptions of Definition \ref{filtration}  we can associate a family ${\cal M}=\{d_r\}_{r\in\R^+}$ of increasing distance. We get an entropy $h({\cal M},M)$ which depends of $\cal A$ and also of $\cal F$. This entropy will be denoted $h({\cal M},{\cal F},M,d)$.

 In the case of foliation defined by the image of an anchored bundle $(A,M,\sharp)$ then the set $\cal A$ of $A$-paths satisfies the property of compatibility (\ref{adcompatible}). Given any Finsler metric $F$ on $A$ and any Finsler metric $\Phi$ on $M$ then the Finsler entropy $h(A,M,\sharp,F,\Phi)$ of $(A,M,\sharp)$ is nothing but else that the previous entropy $h({\cal M},{\cal F},M,d_\Phi)$ if $d_\Phi $ is the distance on $M$ associated to $\Phi$ (see section \ref{entropyD})

\section{Finsler entropy of a  smooth  distribution}
\subsection{Smooth  distribution}\label{smoothD}${}$\\
We first begin by some preliminaries:

Consider a connected paracompact manifold $M$ of dimension $n$ and let $p: A\ap M$ be a smooth real vector bundle over $M$. 

1. Denote by ${\cal C}^\infty(M)$ the algebra of smooth  functions on $M$ and  in this algebra consider 
 ${\cal C}_c^\infty(M)$ the ideal  of smooth functions with compact support.

2. Denote by ${\cal C}^\infty(M,A)$ the ${\cal C}^\infty(M)$-module of smooth sections of  vector bundle $p: A\ap M$ and in this module consider   
the submodule   ${\cal C}_c^\infty(M,A)$  of ${\cal C}^\infty(M)$ of smooth sections with compact support in $M$. In particular the module ${\cal C}_c^\infty(M,TM)$ will be denoted ${\Xi}(M)$. 

3. Let $\cal E$ be a submodule of ${\cal C}_c^\infty(M,A)$. The submodule $\hat{\cal E}\subset {\cal C}^\infty(M,A)$ of global sections of $\cal E$ is the set
of sections $\s\in {\cal C}^\infty(M)$ such that, for all $\phi\in  {\cal C}_c^\infty(M)$, we have $\phi.\s\in  {\cal E.}$\\
\noindent If $\g:I\ap M$ is a smooth curve defined on an interval $I$ of $\R$,  we denote by ${\cal E}_\g$ the restriction of $\cal E$ to $\g(I)$ and any $\s\in {\cal E}_\g$ is called {\it a section of $\cal E$ along $\g$}.

The module $\cal E$ is said to be {\it finitely generated } if there exist global sections $\s_1\dots\s_k$ of $\cal E$ such that
$${\cal E} = {\cal C}_c^\infty(M)\s_1 +\cdots+ {\cal C}_c^\infty(M)\s_k.$$

4. Let  $f : N \ap M$ be a smooth map between two manifolds $N$ and $M$. Denote by $f^*(A)$ the pull-back bundle on $N$ of a vector bundle $A\ap M$ over $M$. If $\cal E$ is a submodule of ${\cal C}_c^\infty(M,A)$, the pull-back module $f^*({\cal E})$ is the submodule of ${\cal C}_c^\infty(N,f^*(A))$  generated $\phi.(\s\circ f)$ with $\phi\in {\cal C}_c^\infty(N)$ and $\s\in {\cal E}$. 

5. A submodule ${\cal E}$ of ${\cal C}_c^\infty(M)$  is said to be {\it locally finitely generated} if there exists an open cover $(U_i)_{i\in I}$
of $M$ such that the restriction of ${\cal A} $ to each $U_i$ is finitely generated.

\begin{Def}\label{singditribution}${}$
\begin{enumerate}
\item A   distribution  on $ M$   is a field $x\mapsto D_x$ of vector subspace of $T_xM$ 
\item A subset ${\cal X}$ of $\Xi(M)$  generates a distribution $D$ if
$$D_x=\textrm{span}\{X_{x}, X\in {\cal X}\}.$$
\item If $\cal D$ a locally finitely generated submodule of $\Xi(M)$,   we will say that $D$ is  locally finitely generated by $\cal D$.
\item $D$ is called a smooth  distribution if there exists a subset ${\cal X}$ of $\Xi(M)$ which generates $D$. 
\end{enumerate}
\end{Def}
If $D$ is a smooth  distribution, the set of all vector fields $X\in\Xi(M)$ which are tangent to $D$ is a submodule of $\Xi(M)$ is denoted ${\cal X}_D$.




The distribution is called regular if $\dim D_x$ is constant (independent of $x$). Then $D$ is a subbundle of $TM$ and  ${\cal D}={\cal C}_c(M,D)$. When the dimension of $D_x$ is not constant we say that $D$ is a singular. An important general situation of smooth distribution which  is  locally finitely generated is when $D$ is the range  of an anchored bundle $(A,M,\sharp)$ ({\it cf.} subsection \ref{anchoredentrypy}) . 

 However, not all smooth distribution $D$  is finitely generated and not all smooth finitely generated distribution is  the image of the  module of section of an anchored bundle (for an illustration of such a situations see examples in \cite{Su} and \cite{AnSk} respectively). But when $M$ is compact we have
 
 \begin{Pro}\label{imagebundle} ${}$\\ 
 Let $D$ be a smooth distribution on a  compact  manifold which is locally finitely generated by a  module $\cal D$. Then there exists an anchored bundle $(A,M,\sharp)$ such that  ${\cal D}=\sharp({\cal C}_c^\infty(M,A))$
\end{Pro}

\begin{proof} Let $x$ be a point in $M$, using the  arguments in the proof of  Proposition 1.5 in \cite{AnSk},  we can show that there exists   $X_1,\dots X_{p(x)}\in  {\cal D}$ and an open neighbourhood $U$ of $x$ in $M$ such that $\cal D$ is generated by $\{X_1,\dots X_p\}$ on $U$ and is a basis of ${\cal D}_x$. Since $M$ is compact, we can find a finite set $U_1,\dots U_N$ of such open sets which is a covering  of $M$. Denote by $\{X_1^i,\dots X_{p_i}^i\}$  a family of vector fields in $\cal D$ which has the previous properties  on $U_i$. Then we get a family ${\cal X}=\dis\bigcup_{i=1}^N\{X_1^i,\dots X_{p_i}^i\}$ of vector fields on $M$. Consider the trivial bundle $A=M\times \R^{p_1+\cdots+p_n}$ and $\sharp: A\ap TM$ the map characterized by $\sharp(x, e_j^i)=X_i^j(x)$ where $\{e_j^i\}_{j=1\dots,p_i}$ is the canonical basis of  the factor $\R^{p_i}$ of $ \R^{p_1+\cdots+p_n}$ for all $i=1,\dots,N$. Now given any $X\in \cal D$, consider an partition of unity $\{\phi_i\}_{i=1,\dots,N}$ associated with $\{U_i\}_{i=1,\dots,N}$. Then by construction of the family ${\cal X}$ ,  on each $U_i$ we have: 
$$\phi_i X=\dis\sum_{j=1}^{p_i }f_j X_j^i=\dis\sum_{j=1}^{p_i} f_j \phi_iX_j^i$$
It follows that
$$X=\dis\sum_{i=1}^N\phi_i X=\dis\sum_{i=1}^N\dis\sum_{j=1}^{p_i} (f_j \phi_i)X_j^i=\sharp(\dis\sum_{i=1}^N\dis\sum_{j=1}^{p_i} (f_j \phi_i)e_j^i)$$
This ends the proof since for all $j=1,\dots,p_i$ and $i=1,\dots,N$ each component $f_j \phi_i$ is a smooth map on $M$  whose support is contained in   the compact support of $\phi_i$.\\

\end{proof}

An important case of smooth distribution is the case of integrable distribution:

\begin{Def}\label{foliation}${}$\\
A smooth  distribution $D$    is called integrable if there exists a partition $\cal F$ of $M$  in embedded smooth manifolds  called leaves such that for each $x\in M$ we have $D_x=T_xL$ for all $x$ in a leaf $L$ ant any leaf $L$ of $\cal F$.
 \end{Def}
If $D$ is integrable, we will say that {\bf $\cal F$ is defined by $\cal D$}. 
\begin{Rem}\label{smoothstef}${}$\\
From \cite{St} or \cite{Su}, if a smooth distribution $D$ is integrable, then the associated foliation is a Stefan-Sussmann foliation. Conversely, if $\cal F$ is a stefan-Sussmann foliation, from Definition \ref{steffol} it follows that the distribution $D$ defined by $D_x=T_xL$ if $L$ is the leaf through $x$ is a smooth distribution generated by the submodule ${\cal X}_{\cal F}$ of $\Xi(M)$ of vector fields in $\Xi(M)$ which are tangent to the leaves of $\cal F$.\\
\end{Rem}

For locally finitely generated smooth distribution we have the classical criteria of integrability:

\begin{The}\cite {St},\cite{Su}:
Let $D$ be a smooth distribution locally finitely generated by a module $\cal D$. if $\cal D$ is stable by Lie bracket of vector fields then $D$ is integrable
\end{The}
Note that the converse is not true. The reader could find contre-examples in \cite{St} and \cite{Su}.\\

The foliation $\cal F$ is called {\bf regular}    if $\dim D_x$ is regular and so $D$ is a subbundle of $TM$. When the dimension of $D_x$ is not constant we say that $\cal F$ is a Stefan-Sussmann  foliation. 

Of course  if $L$ is the leaf of $\cal F$  through $x$ then $T_yL=D_y$ for any $x\in L$. In particular $D_x$ has constant dimension on $L$

An important  example of integrable smooth  distribution is given by  a {\bf foliated anchored bundle} (see \cite{Pe}). More precisely,
A foliated anchored  bundle is an anchored bundle $(A,M,\sharp)$ such that ${\cal D}=\sharp({\cal C}_c^\infty(M,A))$ is stable under  Lie bracket. \\
Recall that an almost Lie bracket on  $(A,M,\sharp)$ is an $\R$ skew symmetric bilinear $[\;,\;]_A: {\cal C}^\infty(M,A)\times  {\cal C}^\infty(M,A)\ap  {\cal C}^\infty(M,A)$ such that

$[\sigma,f\rho]_A=f[\sigma,\rho]_A+df(\sharp\s)\rho$ for all $\s,\rho\in  {\cal C}^\infty(M,A)$ and $f\in  {\cal C}^\infty(M)$

 There always exists an almost Lie bracket $[\;,\;]_A$ on  any  anchored bundle $(A,M,\sharp)$. If $(A,M,\sharp)$ is an anchored vector bundle, there always exists an almost Lie bracket  $[\;,\;]_A$  which is {\it compatible with the anchor} $\sharp$ that is:

 $[\sharp\s,\sharp\rho]=\sharp[\s,\rho]_A$ for all $\s,\rho\in  {\cal C}^\infty(M,A)$.
 
 This situation occurs in particular when $(A,M,\sharp),[\;,\;]_A)$  is a Lie algebroid   that is $( {\cal C}^\infty(M,A),[\;,\;]_A)$ has a Lie algebra structure and then $\sharp:  {\cal C}^\infty(M,A)\ap  {\cal C}^\infty(M,TM)$ is a Lie algebra morphism.
 
 However, generally the distribution $D$ associated to a  Stefan-Sussmann foliation $\cal F$  is not locally finite generated (see Examples in \cite{St} or \cite{Su}). Moreover even if this distribution $D$ is finitely generated this not implies that $\cal F$  comes from a foliated anchored bundle (see an Example in \cite{AnSk}).  But when $M$ is compact, from Proposition \ref{imagebundle}
  we have
 
 \begin{Pro}\label{image foliated} ${}$\\ 
 Let $\cal F$ be a   foliation associated to an integrable smooth locally finitely generated distribution $D$ on a smooth compact  manifold. Then there exists a foliated anchored bundle $(A,M,\sharp)$ such that  ${\cal F}$ is defined by the module $\sharp({\cal C}_c^\infty(M,A))$.\\
\end{Pro}

\begin{Def}\label{transverse}${}$\\
Let  $f : N \ap M$ be a smooth map between two manifolds $N$ and $M$ and $D$ a smooth  distribution on $M$ and  ${\cal D}$ be the subbmodule of $\Xi(M)$  which generates $D$.
\begin{enumerate}
 \item Denote by  $f^{-1}({\cal D}) = \{X \in \Xi(N)\;:\; Tf(X) \in f^*({\cal D})\} $ the submodule of $\Xi(N)$.\\
 \item  We say that $f$ is transverse to $D$ 
 if for any $x\in M$, we have $D_{f(x)}+Tf(T_xN)=T_{f(x)}M$
 \end{enumerate}
 \end{Def}
 For example if $f$ is a submersion then $f$ is transverse to any smooth distribution on $M$. 
 
 \begin{Pro}\label{pullback}${}$\\
 Let  $f : N \ap M$ be a smooth map between two manifolds $N$ and $M$  and $D$ be a smooth  distribution  on $M$ generated by a module $\cal D$.
 \begin{enumerate}
 \item If $D$ is generated by $\cal D$, then  $f^{-1}({\cal D})$  generates a smooth  distribution $f^{-1}(D)$ on $N$. Moreover, if $f$ is transverse to $D$ then $\textrm{codim }\{f^{-1}(D)\}_x=\textrm{codim }D_{f(x)}$
 \item  If $f$ is tranverse to $D$ and if $D$ is integrable so is  $f^{-1}(D)$ and any leaf $\tilde{L}$ of this foliation  is a connected components of $f^{-1}(L)$ of where $L$ is some leaf of the foliation associated to $D$ and $\tilde{L}$ and $L$ have the same codimension.
 \end{enumerate}
 \end{Pro}
 Under the assumptions of Proposition \ref{pullback}, the distribution $f^{-1}(D)$ is called the {\bf pull back distribution} of the distribution $D$. If moreover $D$ is integrable and $\cal F$ is the associated foliation, the foliation $f^{-1}({\cal F})$ associated to $f^{-1}(D)$ is called {\bf the pull back foliation} of the foliation $\cal F$.
 
 \begin{proof}[Proof of Proposition \ref{pullback}]${}$\\
We can consider the distribution $\D$ on $N$ defined by $\D_x=(T_xf)^{-1}(D_{f(x)})$. Assume that $D$ is generated by a module $\cal D$. From the definition of  $f^{-1}({\cal D})$  it follows easily that  $\D$ is generated by $f^{-1}({\cal D})$.  If $f$ is transverse to $D$  then it is 
clear that codim $\D_x=$codim $D_{f(x)}$\\
Now assume moreover that $D$ is integrable. Since $f$ is transverse to $D$, then $f$  is in particular transverse to any leaf $L$ of the foliation $\cal F$ defined by $D$. Given any $x\in N$, there exists a open neighborhood $V$ of $f(x)$ in $M$ such that $V\cap L$ is an embedded submanifold of $M$ if $L$ is the leaf through $f(x)$. Therefore  each connected component $U_L=f^{-1}(V\cap L)$ is an integral manifold of $\D$ ({\it i.e.} $\D_z=T_zU_L$ for any $z\in U_L$).  Since $\D$ is smooth, it follows by classical arguments ({\it cf.} \cite{Su}) that $\D$ is integrable then the result follows immediately



 \end{proof}
\subsection{Finsler metric of a smooth singular distribution}${}$\\

{\it We begin by defined the concept of  weak Finsler metric on a smooth singular distribution:}\\

\begin{Def}\label{Finslerdistribution}${}$\\
 Let $ D$ be smooth distribution 
 on a  manifold $M$.
\begin{enumerate}
\item A {\bf weak  Finsler metric on D}  is a map $F:{D}\ap \R^+$ such that:

(i) For any absolutely  curve $\g:[a,b]\ap M$  tangent to $D$ {\it a.e.}   the map $t\mapsto F(c(t),\dot{c}(t))$  is a measurable  map from $ [a,b]$ to $\R$ bounded {\it a.e.};

(ii) $F$ induces a Minkowski norm on each vector space $D_x$ for any $x\in M$.
\item A  weak  Finsler metric $F$ on $D$ is called  of class $C^k$  ($0\leq k\leq \infty$) if for any $x\in M$,  any  absolutely continuous curve   $\g:]-\a,\a[\ap M$  tangent to $D$ {\it a.e.} of  $C^k$ such that $\g(0)=x$  and any vector field $X\in \{{\cal X}_D\}_\g$, along $\g$,   the map $t\mapsto F(c(t)X_{\g(t)})$  is  map from $ ]-\a,\a[$ to $\R^+$ of class $C^0$ for $0\leq k\leq \infty$ and moreover for $1\leq k\leq \infty $
 if $X_{\g(t)}\not=0$ for all $t\in ]-\a,\a[$ then $t\mapsto F(c(t),X_{\g(t)})$ is of class $C^k$
\end{enumerate}
\end{Def}

The  field  of quadratic forms $(x,u) \ap g_x(u)$   associated to $F(x,\;)$ will be denoted $g_F$.
Note that if $F$ is a Finsler metric of class $C^k$ for any $1\leq k\leq \infty$, then the property (i) is automatically satisfied.\\
 Given a weak Finsler metric on $D$, for any absolutely continuous curve $\g:[a,b]\ap M$ tangent to $D$  we can define its length:
\begin{equation}\label{length}
l(\g)=\dis\int_a^b F(\g(t),\dot{\g}(t))dt
\end{equation}
As classically the length of  such a curve is independent of the chosen  parametrization. In particular we can always assume that $\g$ is defined on $[0,1]$. Any absolutely continuous curve $\g$ defined on  some interval $ [a,b]$ and tangent to $D$ {\it a.e.} will be called {\bf a $D$-admissible curve}.
According to the classical result of accessibility of \cite{Su}, there exists a Stefan-Sussmann foliation $\cal F$ such that  each leaf $L$ of this foliation is  an equivalence relative of the following equivalence relation:\\
\begin{eqnarray}\label{Requiv}
x \equiv y \Longleftrightarrow   \textrm{ there exists a $D$-admissible  curve } \g:[0,1]\ap M \textrm{ with } \g(0)=x \textrm{ and } \g(1)=y.
\end{eqnarray}
 This foliation $\cal F$ will be called the {\bf accessibility foliation} of $D$. On each leaf $L$ of $\cal F$ we have a distance $d_L$ which is defined by
\begin{eqnarray}\label{dL}
d_L(x,y) =\inf\{ l(\g)\;,\;\g:[0,1]\ap L \textrm{ which is  $D$-admissible,  with } \g(0)=x \textrm{ and } \g(1)=y\}.
\end{eqnarray}

\begin{Rem}\label{horizontalcomp}${}$\\
It is clear that the  set $\cal A^D$ of $D$-admissible curves for $D$  satisfies the assumptions of Definition \ref{admissible}  and also  the compatibility  property (\ref{adcompatible}). It follows that $\cal A^D$ is a set of admissible curves which is   compatible with the accessibility foliation $\cal F$ of $D$.\\
\end{Rem}
When $D$  defines  a  Stefan-Sussmann foliation $\cal F$, then each leaf $L$ is an  equivalence class for the equivalence relation (\ref{Requiv}). Now a Finsler metric $F$ of class $C^k$ on $D$ induces on each leaf $L$
a Finsler metric $F_L$  of class $C^k$ and so also a distance $d_L$ on $L$ which is exactly the distance defined in (\ref{dL}).

\begin{Exs}\label{exsingfol}${}$

{\bf 1}. Consider any Finsler metric $\Phi$ on $M$ of class $C^k$ and $D$ a smooth  distribution. Then $\Phi$ induces a Finsler  metric $F$  on the distribution $D$ of class $C^k$.\\

{\bf 2} Let $\cal F$ be a Stefan-Sussmann  foliation on $M$ defined by an integrable distribution $D$. Assume that we have a Finsler metric  $F_L$ of class  $C^k$, $k\geq 1$, on  each leaf $L$ of $\cal F$. Then the "collection " $\{F_L, L \textrm{ leaf of } {\cal F}\}$  gives rise to a  Finsler metric on the distribution $D$ generated by $\cal F$ {\it again denote $F$} defined by $F(x,u)=F_L(x,u)$ if $L$ is the leaf of $\cal F$ through $x$. Indeed  the properties (ii) and (iii) of Definition \ref{Finslerdistribution} are clearly satisfied. Now any 
  curve $\g:]-\a,\a[\ap M$ of class $C^k$ which is tangent to $D$  must be contained in one leaf $L$ and  if $c(0)=x$, any vector field $X$ of ${\cal F}_x$ along $\g$ is also tangent to $L$. Therefore  the assumptions  of Definition \ref{Finslerdistribution} point (2) are also satisfied.\\

{\bf 3}.  Consider a foliated  anchored bundle $(A,M,\sharp)$ and ${\cal D}=\sharp({\cal C}^\infty_c(M,A))$. Any Finsler metric $F$ on $A$ induces 
on each leaf $L$ of $\cal F$   a Finsler metric denoted $F_L$ (see \cite{Pe} Example 8.1.1.(4)). The "collection " $\{F_L, L \textrm{ leaf of } {\cal F}\}$  defines a Finsler metric of class $C^\infty$ on $D$  according to the previous example.\\

{\bf 4} An important particular case of a foliated anchored bundle is the situation of a Poisson structure. Recall ({\it cf.} \cite{Ma} for example) that a Poisson structure on a manifold $M$ is the data of a Poisson  bracket $\{\;,\;\}$ on the algebra ${\cal C}^\infty(M)$. It is equivalent to the data of a morphism $P:T^*M\ap TM$ such that $<df,Pdg>=\{f,g\}$(for more explicit details see for example  \cite{Ma}). Then the range $D=P(T^*M)$ of $P$ is an integrable distribution and the associated foliation $\cal F$ is Stefan-Sussmann foliation. If we put a Finsler metric  $\Phi$  on $TM$ we obtain a Finsler metric $\Phi^*$ on $T^*M$ and as in the previous example, we get a Finsler metric $F_L$ on each leaf $L$ of $\cal F$

\end{Exs}

 Unfortunately Example \ref{exsingfol} {\bf 3} can not be generalized to the case of distribution $D=\sharp(A)$  where  $(A,M,\sharp)$ is an anchored bundle.  We only have
 
 \begin{Pro}\label{weakfinsleranchor}${}$\\
 Consider an  anchored bundle $(A,M,\sharp)$ and $D=\sharp(A)$. A Finsler metric $F$ on $A$ induces a map
 \begin{center}
  $F_D: D\ap \R^+$ by
${F}_D(x,v)=\inf\{F(x,u): v=\sharp(x,u), u\in A_x\}$ 
\end{center}
which is a weak Finsler metric on $D$.
\end{Pro}
\begin{proof}
Note that we have ${F}_D(x,v)= F(x,u_0)$  for some $u_0$ in the fiber $p^{-1}(x)$.\\
 Indeed  set  $r_0= \inf\{F(x,u): v=\sharp(x,u), u\in A_x\}$. Now  if $\sharp(x,u_0)=v$ then $\sharp^{-1}(v)=u_0+\ker \sharp(x\,;)$ and  then $r_0=F(x,u_0)$ since $ \inf\{F(x,u)\;:\;u\in \ker \sharp(x,\;)\}=0$. 
 Therefore we get  $F_D(x,v)=F(x,u_0)$.
It follows that $F_D$ satisfies Property (i) and (ii)  of Definition \ref{Finslerdistribution}.  The following Lemma will complete the proof:

\begin{Lem}\label{Existenceconicalcontrol}${}$\\ 
Consider a smooth singular distribution $D=\sharp(A)$  where  $(A,M,\sharp)$ is an anchored bundle and $F$ a Finsler metric on $A$. For any $A$-admissible curve $\g:[0,1]\ap M$ there exists a measurable curve $({\g},u): [0,1]\ap A$ such that $\sharp(\g(t),u(t))=\dot{\g}(t)$ {\it a.e.} and  $F_D(\g(t),\dot{\g}(t)=F(\g(t), u(t))$
\end{Lem}

From this Lemma   since $F$ is a smooth Finsler metric and $\g$ is absolutely continuous it follows that $t\ap  F_D(\g(t),\dot{\g}(t))$ is a measurable map bounded {\it a.e.}.\\
\end{proof}
\begin{proof}[Proof of Lemma \ref{Existenceconicalcontrol}]${}$\\
As in  the proof of Theorem 5.3 in \cite{PeVa} we have a decomposition  $[0,1]=\dis\bigcup_{\theta\in \Theta}I_\theta$ into disjoints semi-interval $I_\theta\subset [0,1]$ such that the pull-back of $A$ over the restriction of $\g_\theta$ of $\g$ to $I_\theta$ is a trivial bundle $A_\theta$ and the  kernel $K_\theta$ of the anchor $\sharp$  in $A_\theta$ has a constant rank. It follows that we can find a  subbundle $H_\theta$ of $A_\theta$ such that  $A_\theta=K_\theta\oplus H_\theta$. Now since the restriction of $\sharp$ to $H_\theta$ is an isomorphism, there exists a measurable curve $u_\theta: I_\theta\ap H_\theta$ such that $\sharp(\g_\theta,u_\theta)=\dot{\g}_\theta$ on $I_\theta$ {\it a.e.}. Moreover,  as we have seen previously, we have $F(\g_\theta,u_\theta)=F_D(\g_\theta,\dot{\g}_\theta)$. In this way, as in \cite{PeVa}, we build a measurable section  $u$ of $A$ over $\g$ such that $\sharp(\g,u)=\dot{\g}$ { \it a.e.} and  $F(\g,u)=F_D(\g,\dot{\g})$.

\end{proof}

\begin{Rem}\label{extension}${}$\\ Consider a  Finsler metric $\Phi$ on $M$. We have seen that $\Phi$ induces a Finsler metric $\Phi_D$ on $D$. Conversely, one can ask when, starting by a Finsler metric $F$ on  a  smooth singular distribution $D$, we can extend the induced Finsler metric $F$  to a global Finsler metric $\Phi$ on $TM$. Of course if $D$ is  regular the answer is positive, but in   general this extension can not exist even if is  we are in the situation of Examples  \ref{exsingfol} {\bf 2} or {\bf 3} or in the context of Proposition \ref{weakfinsleranchor} (see for instance  \cite{PeVa} in the Riemannian context).
\end{Rem}

\subsection{Finsler entropy of a smooth  distribution}\label{entropyD}${}$\\
Let $D$  be a smooth  distribution on a compact manifold $M$. We provide  $D$ with a weak Finsler metric $F$.  We have  seen that the set $\cal A^D$ of $D$-admissible curves  is set of admissible curves ({\it cf.} Remark \ref{horizontalcomp}). We put
$${\cal A^D}_r=\{\g\in {\cal A^D} : l(\g)\leq r\}$$

It is easy to see that $\{{\cal A^D}_r\}_{r\in \R^+}$ is a filtration of $\cal A^D$  {\it i.e.} it satisfies the assumptions of Definition \ref{filtration}. Given any distance $d$ on $M$  and $\{{\cal A^D}_r\}_{r\in \R^+}$ we can define a set ${\cal M}=\{d_r\}_{r\in \R^+}$ of increasing distances. Then we have
\begin{Def}\label{entropydistribution}${}$\\
The entropy $h({\cal A^D},M,d)$ is called the Finsler entropy of $D$ and  is denoted $h(D,M,F,d)$
\end{Def}

 Now, {\bf if $\cal F$ is a Stefan-Sussmann foliation}, given a Finsler metric on each leaf, we get a Finsler metric on the distribution $D$ associated to $\cal F$

\begin{Def}\label{entropycalF}${}$\\
The entropy $h({\cal F},M, F)$ of the foliation   $\cal F$ is the entropy $h(D,M,F)$ of the distribution $D$ associated to $\cal F$.\\
\end{Def}

The entropy of a smooth distribution have the following properties:

\begin{Pro}\label{independentfinsler}${}$\\
Let $\Phi$ a Finsler metric on a compact manifold $M$.  
\begin{enumerate}
\item  The value of $h(D,M,F,d_\Phi)$ is invariant by bi-Lipschitz homeomorphism of the metric space $(M,d_\Phi)$ where $d_\Phi$ is the distance associated to $\Phi$. Moreover,  the value of $h(D,M,F,d_\Phi)$ is independent of the choice of such a Finsler metric $\Phi$.
\item for $\l>0$, we have $h(D,M, \l.F,d)=\l^{-1}. h(D,M, F,d)$
\item   Assume that we have two weak Finsler metric $F$ and $F'$ on $D$ such that
$$ F'\leq C.F$$
for some $C>0$ . Then we have  $$h(D,M,F',d_\Phi)\geq C^{-1} h(D,M,F,d_\Phi).$$
\item  Let $\cal F$ be the accessibility  foliation of $D$. Then we have  $h(D,M, F,d_\Phi)=h({\cal M},{\cal F},M,d_\Phi)$ \\({\it cf.} subsection \ref{entropiefol}). 
\end{enumerate}
\end{Pro}

{\bf We have the first following results about the entropy of a smooth distribution:}\\


\begin{The}\label{Zung}${}$\\
Let $F$ be a Finsler metric on an   anchored bundle $(A,M,\sharp)$ and $D=\sharp(A)$. If $F_D$ is the weak Finsler metric induced on $D$ by $F$ ({\it cf.} Proposition \ref{weakfinsleranchor}) then we have:
$$h(A,M,\sharp,F,\Phi)=h(D,M,F_D)$$
(for the definition of $h(A,M,\sharp,F,\Phi)$ see subsection \ref{anchoredentrypy}). \\ 
\end{The}
\begin{The}\label{h=0} {\bf or  Theorem 1} in the introduction${}$\\
Assume that the distribution $D$ is  controllable {\it i.e.} the accessiblity foliation $\cal F$ has only one leaf equal to $M$. Then the entropy $h(D,M, F)$ is zero.\\
\end{The}
\begin{Rem}\label{refZu}${}$\\
 Let $D$ be a regular distribution on $M$ and $g_D$ be the  Riemannian metric on $D$ induced  by  a given Riemannian metric $g$ on $M$. Then the corresponding entropy $h(A,M,\sharp,F)$ is exactly the entropy of the distribution $D$ defined in \cite{Bis} ({\it cf.} subsection \ref{anchoredentrypy}). Therefore Theorem \ref{Zung} and Theorem \ref{h=0} imply that if $D$ is a contact distribution  this entropy is zero as it is already proved in \cite{Bis}. More generally Theorem \ref{Zung} and Theorem \ref{h=0} also imply Theorem 3.4 and Theorem 3.6 which are particular cases where the distribution $D=\sharp(A)$ is controllable.
 \end{Rem}

 According to the previous  results we have the following consequences:
\begin{Obs}\label{propentrop}\end{Obs}
\begin{description}
\item[(1)]  The {\it Finsler entropy of a smooth singular distribution  $D$ on $M$ is independent of $\Phi$} and will be denoted $h(D,M, F)$.\\

\item [(2)] if $F$ and $F'$ are two (weak) Finsler metrics on $D$ which are  {\bf equivalent } in the sense that there exists a constant $C>0$ such that 
$$C^{-1}.F\leq F'\leq C. F$$
then we have 
$$ C^{-1}. h(D,M,F,d_\Phi)\leq h(D,M,F',d_\Phi)\leq C.h(D,M,F,d_\Phi)$$
Thus the fact that  the entropy $h(D,M, F)$ is zero, finite or eventually infinite\footnote{we do not know if the Finsler entropy of a smooth distribution is always finite} depends only of the equivalence class of the Finlser metric $F$ on $D$. In particular, if  the Finsler metric $F$ is induced by $\Phi$, then  the property  to be   zero, finite or infinite of the corresponding entropy    is an intrinsic property  of $D$. This entropy will be called the geometric entropy of $D$.\\

\item[(3)]  if $F$ and $F'$ are two(weak) Finsler metrics on $D$ which are not equivalent in the sense of {\bf (2)} then $h(D,M,F)$ can be zero and $h(D,M,F')$  can strictly positive. This situation can be illustrated by the following situation.

 From \cite{SYZ} there exists a vector field $Z$ on a compact manifold $M$ and positive smooth  functions $\varphi$ and $\varphi'$ on $M$ such that $\varphi(x_0)=\varphi(x_0)=0$ for some point $x_0\in M $, strictly positive for $x\not=x_0$, which are equal on the complementary of a small neighborhood of $x_0$ and such that the topological entropy of the flow of $\varphi.Z$ and  $\varphi'.Z$ is zero and strictly positive respectively. Consider the trivial bundle  $A=M\times \R$ over $M$ , 
$\sharp: A\ap TM$  and $\sharp':A\ap TM$ the anchor characterized by $\sharp(e)=\varphi.Z$ and $\sharp'(e)=\varphi.Z_x$ respectively, where $e: M\ap A$ is the canonical section whose value is the canonical basis of $\R$. Then the range of $\sharp$ and $\sharp'$
is the same  distribution $D$ generated by  $\varphi.Z$ or  $\varphi'.Z$ and $x_0$ is the unique singularity. Now, if  we put on $A$ the Finsler metric given by the canonical absolute value $|\;|$ on $\R$, we get on $D$ two weak Finsler metric $F_D$ and $F'_D$  ({\it cf.} Proposition \ref{weakfinsleranchor}). These Finsler metric     cannot be compared on a small neighborhood of $x_0$. On the other hand From Theorem  3.7 in \cite{Zu} and Theorem \ref{Zung} the entropy $h(D,M,F_D)$ is twice the entropy of the flow of $\varphi.Z$ and $h(D,M,F_D)$ is twice the entropy of the flow of $\varphi'.Z$ respectively.\\

\item[(4)] The Finsler entropy of a distribution $D$ is nothing but else the entropy of the accessibility foliation of $\cal F$ associated to the accessible set of $D$-curves. However in general this entropy  {\bf is different from}   the entropy of the distribution defined  this accessibility foliation $\cal F$ even if we  a Finsler extension  $\hat{F}$ to  $\cal F$ of the given Finsler metric $F$ on $D$. Of course when $D$ is integrable and $\cal F$ the associated accessibility foliation  of $D$ and the previous entropy is exactly the entropy of $\cal F$.\\

\item[(5)] According to Theorem \ref{Zung} and section \ref{folentrop} the geometrical entropy of a regular distribution is exactly the geometric entropy of a foliation as previously defined in {\bf 2} where we take for $\Phi$ a Riemannian metric on the manifold and  and for $F$ the Riemannian metric induced on  tangent bundle of the foliation. In particular in this case this entropy either zero or finite. On the other hand if $\cal F$ is the foliation associated to a Poisson structure on $M$ ({\it cf.} Example \ref{exsingfol} {\bf 4}), again from Theorem \ref{Zung} the geometric entropy of $\cal F$ is exactly the entropy of $\cal F$ as defined in \cite{Zu}. In particular the reader will find in this paper an example of such a foliation  whose entropy is finite but not zero.\\

\item[(6)] We can define a notions of entropy for a smooth distribution $D$ on a non compact manifold $M$. Let $K$ be a relatively compact subset of complete Finsler manifold $(M,g)$. As previously, given a (weak) Finsler metric on $D$, we can define a family ${\cal M}_K=\{d_r\}_{r\in \R^+}$ on  in the closure $\bar{K}$  associated to the set of absolutely curves tangent to $D$  contained in $\bar{K}$ with length at most $r$. Therefore the entropy $h(D,K,F,\Phi)=h({\cal M}_K,K,d_\Phi)$ is well defined (see Remark \ref{compactX} Point (1)). More generally if we  consider a sequence 
$$K_1\subset K_2\subset\cdots\subset K_n\subset\cdots \subset M$$
of relatively compact subsets such that $M=\dis\bigcup_{n\in \N}K_n$. Then given a (weak) Finsler metric on a smooth distribution $D$ and a Finsler metric  on $M$ we have
$$h(D,M,F,d_\Phi)=\dis\lim_{n\ap \infty}h(D,K_n,F,(d_\Phi)).$$ It is easy to see that this definition does not depends of the choice of the sequence $(K_n)$.\\
\end{description}
\bigskip

{\bf We end this section by the proof of the announced results}

\begin{proof}[Proof of Proposition \ref{independentfinsler}]${}$\\
\noindent { \it Point (1)}: consider a metric $d'$ on $M$ which is bi-Lipschitz equivalent to $d_\Phi$.  There exists a constant $C>0$ such that $\dis\frac{1}{C}d\leq d'\leq Cd$.  If $\{d'_r\}_{r\in \R^+}$ is the family of distances associated  $\{{\cal A^D}_r\}_{r\in \R^+}$ relative to $d'$ clearly we also have $\dis\frac{1}{C}d_r\leq d'_r\leq Cd_r$ for any $r\in \R^+$. 
Then the first property  is just an application of    Remark \ref{compactX} Point (2). Now since $M$ is compact any two Finsler metric $\Phi$ and $\Phi'$ the associated distances $d_\Phi$ and $d_{\Phi'}$ are bi-Lipschitz equivalent.\\


 \noindent  {\it Point  (2)}:   we can use the same arguments used in the proof of Proposition 3.1 in \cite{Zu}.\\

 \noindent  {\it Point  (3)}: Assume that we have $F'\leq C. F$.\\
  Denote by $l(\g)$ and $l'(\g)$ the length of a curve $\g\in {\cal A}$ relative to $F$ and $F'$ respectively. If $\g$ is defined on $[0,1]$ from our assumption we have  $  l'(\g)\leq C. l(\g)$.  Let $\{{{\cal A}'^D}_r\}_{r\in\R^+}$ the filtration of $\cal A^D$ associated to $F'$. 
  Therefore we obtain $${\cal A^D}_{r}\subset {\cal A'^D}_{Cr}.$$
 We set 
 $$\d'_r(x,y)=\dis\sup_{\g'\in{{\cal A}'^D}_r(x)}\bar{d}(\g',{{\cal A}'^D}_r(y)).$$
 Now in ${\cal C}^0([0,1], X)$  for any bounded subset $X$ and $Y$ of ${\cal C}^0([0,1], X)$ we set 
 $$\bar{\d}(X,Y)=\dis\sup_{x\in X}\bar{d}(x,Y)$$
On the one  hand, for an $\g\in {\cal A^D}_r(x)$ we have

$\bar{d}(\g,{\cal A^D}_r(y))\leq \bar{d}(\g,{\cal A'^D}_{Cr}(y))$ and so $\bar{\d}({\cal A^D}_r(x),{\cal A^D}_r(y))\leq \bar{\d}({\cal A^D}_r(x),{\cal A'^D}_{Cr}(y))$

On  the other hand we also have

$\bar{\d}({\cal A^D}_r(x),{\cal A'^D}_{Cr}(y))=\dis\sup_{g\in{\cal A^D}_r(x)}\bar{d}(\g,{\cal A'^D}_{Cr}(y))\leq \dis\sup_{g\in{\cal A'^D}_r(x)}\bar{d}(\g,{\cal A'^D}_{Cr}(y))=\bar{\d}({\cal A'^D}_r(x),{\cal A'^D}_{Cr}(y))$

 Finally we get 
 $$\d_r(x,y)=\bar{\d}({\cal A^D}_r(x),{\cal A^D}_r(y))\leq \bar{\d}({\cal A'^D}_r(x),{\cal A'^D}_{Cr}(y))=\d'_{Cr}(x,y)$$
 
 We easily obtain the inequality
  $$h(D,M,F',d_\Phi)\geq C^{-1} h(D,M,F,d_\Phi).$$
 
 

 \noindent {\it Point (4)}: clearly the set of admissible curves ${\cal A^D}$ satisfies the assumption of condition (\ref{adcompatible}). This ends the proof.\\
\end{proof}

For the proof of the Theorems we need an auxiliary result.\\

 We set
$${\cal N}_r=\{\g\in {\cal A}\textrm{ such that } F(\g(t),\dot{\g}(t))\leq r \;{\it a.e.}\}.$$ 
It is clear that $\{{\cal N}_r\}_{r\in \R^+}$ is a filtration of $\cal A$. Moreover we have

\begin{Lem} \label{distN}${}$\\
Let $\{\nu_r\}_{r\in \R^+}$ the family of increasing distances associated to the filtration $\{{\cal N}_r\}_{r\in \R^+}$. Then we have $\n_r=d_r$ for all $r>0$
\end{Lem}

\begin{proof} According to Remark \ref{interpretationdelta} it is sufficient to prove the relation
$$\d_r(x,y)= \dis\sup_{\g\in{\cal N}_r(x)}\bar{d}(\g,{\cal N}_r(y)).$$
 Clearly ${\cal N}_r(x)$ is contained in ${\cal A^D}_r(x)$. Now if $ \g$ belongs to ${\cal A^D}_r(x)$, the arc-length parametrization of $\g$ gives rise to a curve $\bar{g}:[0,l(\g)]\ap M$   such that $F(\bar{\g}(t),\dot{\bar{\g}}(t)=1$ {\it a.e.}. Since $l(\g)\leq r$ the curve $\hat{\g}(t)=\bar{\g}(t.l(\g))$ for $t\in [0,1]$  belongs to ${\cal N}_r(x)$.


 \noindent For any $\g\in {\cal A^D}_r(x)$, since there exists  $\t:[0,1]\ap [0,1]$ such that $\hat{\g}(\t(t))=\g(t)$ and $\hat{\g}\in {\cal N}_r(x)$ then we get

$\bar{d}(\g,{\cal A^D}_r(y))=\bar{d}(\hat{\g},{\cal A^D}_r(y))$.

\noindent Let $\g_n\in{\cal A^D}_r(x)$ be a sequence such that 
  
  $\d_r(x,y)=\dis\sup_{\g\in{\cal A^D}_r(x)}\bar{d}(\g,{\cal A^D}_r(y))=\dis\lim_{n\ap \infty}\bar{d}(\g_n,{\cal A^D}_r(x))$.
  
  \noindent Then we obtain
\begin{eqnarray}\label{drN}
 \d_r(x,y)=\dis\lim_{n\ap \infty}\bar{d}(\hat{\g}_n,{\cal A^D}_r(x))
\end{eqnarray}
.


\noindent Now by the same argument in ${\cal A^D}_r(y)$ we obtain

$\bar{d}({\g},{\cal A^D}_r(y))=\bar{d}({\g},{\cal N}_r(y))=\bar{d}(\hat{\g},{\cal N}_r(y))$.

\noindent It follows that in the one hand

$\d_r(x,y)\geq \dis\sup_{\g\in{\cal N}_r(x)}\bar{d}(\g,{\cal A^D}_r(y))=\dis\sup_{\g\in{\cal N}_r(x)}\bar{d}(\g,{\cal N}_r(y))$.

\noindent On the other hand  from (\ref{drN}) we obtain

$\d_r(x,y)=\dis\lim_{n\ap \infty}\bar{d}(\hat{\g}_n,{\cal N}_r(x))\leq \dis\sup_{\g\in{\cal N}_r(x)}\bar{d}(\g,{\cal N}_r(y))$

\noindent Finally we get

$\d_r(x,y)= \dis\sup_{\g\in{\cal N}_r(x)}\bar{d}(\g,{\cal N}_r(y))$

\end{proof}

\begin{proof}[Proof of Theorem \ref{Zung}]${}$\\
Let $l(\g)$ be the length of a $D$-admissible  curve $\g$ relative to the weak Finsler metric $F_D$. On the one hand, we have the family $\{d_r\}_{r\in\R^+}$ of increasing distance associated to the filtration defined by
$$ {\cal A^D}_r=\{\g\in {\cal A^D} : l(\g)\leq r\}$$
On the other hand, if $\cal A$ is the set of $A$-paths associated to $(A,M,\sharp)$, we have the family $\{d'_r\}_{r\in\R^+}$ of increasing distances associated  to the filtration defined by
$${\cal A}_r=\{\g\in {\cal A} \textrm{ such that } \exists u:[a,b]\ap A \textrm{ over  } \g \textrm{ and }  F(\g(t),{u}(t))\leq r\; \forall t\in [a,b] \;{\it a.e.}\}$$
Therefore it is sufficient to prove that $d_r=d'_r$. \\
From Lemma \ref{distN}, the family $\{d_r\}_{r\in\R^+}$ is also associated to the filtration defined by
$${\cal N}_r=\{\g\in {\cal A}\textrm{ such that } F_D(\g(t),\dot{\g}(t))\leq r \;{\it a.e.}\}.$$ 

On one hand from the definition of $F_D$ and Lemma \ref{Existenceconicalcontrol}, we have  we have ${\cal N}_r(x)\subset{\cal A}_r(x)$. On the other hand if $\g\in {\cal A}_r(x)$, there exists $u:[0,1]\ap A$ over $\g$ such that $F(\g(t),u(t))\leq r$ {\it a.e.} from  the definition of $F_D$ and Lemma \ref{Existenceconicalcontrol} there exists $\hat{u}:[0,1]\ap A$ over $\g$ such that $F_D(\g(t)\hat{u}(t)\leq r$ {\it a.e.} so $\g$ belongs to ${\cal N}_r(x)$. It follows that ${\cal N}_r(x)={\cal A}_r(x)$. It implies directly that $d_r=d'_r$ for all $r\in \R^+$.

\end{proof}

\begin{proof}[Proof of Theorem \ref{h=0} and Theorem 1]${}$\\
Since the accessibility foliation reduces to one leaf  $M$, the distance $d_M$ of minimal length for $D$-admissible curves is a distance on $M$ such that $d_\Phi\leq d_M$. Now, to the distance $d_M$ and the filtration ${\cal N}_{r}$ of $\cal A$ ({\it cf.} Lemma \ref{distN}) we can associate the family ${\cal M}=\{d'_r\}_{r\in \R^+}$ of increasing distances and consider the entropy $h({\cal M}', M,d_M)$. 
According to Remark \ref{interpretationdelta}  and the construction of $d_r$ we have
$$d'_r(x,y)=\dis\sup_{\g\in {\cal N}_r(x)}\bar{d}_M(\g,{\cal N}_r(y))+\dis\sup_{\mu\in {\cal N}_r(y)}\bar{d}_M(\mu,{\cal N}_r(x))$$
where $\bar{d}_M(\g,\mu)=\dis\sup_{t\in[0,1]}d_M(\g(t),\mu(t))$. 
Now for any $\g\in {\cal N}_r(x)$ we have $\bar{d}(\g,{\cal N}_r(y))\leq \bar{d}_M(\g,{\cal N}_r(y))$. Therfore according to Lemma \ref{distN} we get 

$\dis\sup_{\g\in {\cal N}_r(x)}\bar{d}(\g,{\cal N}_r(y))\leq \dis\sup_{\g\in {\cal N}_r(x)}\bar{d}_M(\g,{\cal N}_r(y))$ 

$\dis\sup_{\mu\in {\cal N}_r(y)}\bar{d}(\mu,{\cal N}_r(x))\leq \dis\sup_{\mu\in {\cal N}_r(y)}\bar{d}_M(\mu,{\cal N}_r(x))$ 

Therefore we obtain  $d_r\leq d'_r$ for all $r\in \R^+$.

From Remark \ref{compactX} Point (2) we obtain:
\begin{eqnarray}\label{dominh}
0\leq h(D,M, F)\leq h({\cal M}', M,d_M)
\end{eqnarray}
It remains to show that  $h({\cal M}', M,d_M)=0$.\\
Fix two points 
$x$ and $y$ in $M$ and denote by $\rho=d_M(x,y)$. For any $\epsilon >0$ there exists a $D$-admissible  curve $\eta:[0,1]\ap M$ such that $\eta(0)=y$,  $\;\eta(1)=x$ and $ l(\eta)\leq \rho+\epsilon=\rho'$. After reparametrization of $\eta $ if necessary, (see proof of Lemma \ref{distN}) we can assume that $\eta$ belongs to ${\cal N}_{\rho'}(y)$. Choose some $r>\rho$ and choose $\epsilon$  such that $r>\rho'=\rho+\epsilon$; consider any $\g\in {\cal N}_r(x)$. As in \cite{Zu} proof Lemma 3.3, we define $\mu:[0,1]\ap M$ by

$\mu(t)=\eta(rt/\rho')$ for $0\leq t\leq \rho'/r$ and $\mu(t)=\g(t-\rho'/r)$ for $\rho'/r<t\leq 1$. 

\noindent It is easy to see that $\mu$ belongs to ${\cal N}_r(y)$. Thus we have  $\d'_r(x,y)\leq l(\eta)\leq \rho+\epsilon$ for any  $\epsilon >0$ small enough. Finally we get

$d'_r(x,y)\leq 2d_M(x,y)$ if $r>d_M(x,y)$

\noindent Assume now that $0<r\leq \rho$ and consider $d'_r(x,y)/\rho$. Since the family $\{d'_r\}$ is increasing with $r$, according to the previous result, for any $r'=\rho+\epsilon$, for some $\epsilon>0$, we have:

$d'_r(x,y)/\rho\leq d'_{\rho'}(x,y)/\rho\leq 2$.

\noindent Therefore for any $r$ we obtain $d'_r\leq 2d_M$. Now according to  Proposition \ref{popdr} finally we get $d'_r=2d_M$ which implies trivially that  $h({\cal M}', M,d_M)=0$.

\end{proof}

\section{Entropy of singular Riemannian foliation}
\subsection{Entropy and  Finsler submersion}\label{finslersub}${}$\\
{\it In this subsection we will recall the Definitions and results of \cite{AlDu} on  isometric Finsler submersion.}\\

Given a Minkowski  normed vector space  $(E,\Phi)$ we denote  $B_\Phi$ the the closed unit ball relative to the norm $\Phi$. A surjective linear map $\pi : {E}' \ap E $ between two Minkowski normed spaces $(\{E',{\Phi}')$ and $(E,\Phi)$ is called  an {\it isometric submersion} $\pi(B_{{\Phi}'})=B_{\Phi}$. In this context 
we have
$$\Phi(u)=\inf\{{\Phi}({u}')\;:\; \pi({u}')=u\}$$
In general we have $\Phi(\pi(u'))\leq {\Phi}'(u')$. A vector ${u}'\in {E}'$ is called {\it horizontal} if $\Phi(\pi({u}'))={\Phi}'({u}')$. The set of horizontal vector is a cone called {\it the horizontal cone of ${E}'$}.

\begin{Def}\label{isomsub}${}$
\begin{enumerate}
\item A map $f$ from a Finsler manifold $({M}',{\Phi}')$ to a Finsler manifold $(M,\Phi)$ is called a Finsler isometric submersion if $f:{M}'\ap M$ is a submersion and $T_xf:T_xM'\ap T_{f(x)}M$ is an isometric submersion of Minskowski normed spaces.
\item an absolutely continuous curve ${\g}':[a,b]\ap M$ is called an horizontal lift of an absolutely continuous curve $\g:[a,b]\ap M$ if $f\circ {\g}'=\g$ and $\dot{{\g}'}(t)$ belongs to the horizontal cone of $T_{{\g}'(t)}M'$ {\it a.e.}. 
\end{enumerate}
\end{Def}
Note that if $\g:[a,b]\ap M$ is an immersed $C^1$ curve and $x'$ is some point in $f^{-1}(\g(a))$ then there exists an unique horizontal lift ${\g}':[a,b]\ap M'$ such that $\bar{\g}(a)=x'$ (see \cite{AlDu}).
Given an absolutely  continue curve  $\g:[a,b]\ap M$,  if $
{\g'}:[a,b]\ap M'$ is an absolutely continuous lift of $\g$ ({\it i.e.} $f\circ {\g}'=\g$) then the length of ${\g}'$ is greater or equal to the length of $\g$. Moreover ${\g}'$ is an horizontal lift of $\g$ if and only if ${\g}'$ and $\g$ have the same length.
\begin{Rem}\label{riemansub}${}$\\
Since a Riemannian manifold  is a particular case of  of Finsler manifold,  a map $f$ from a Riemannian manifold $({M}',{g}')$ to a Riemannian manifold $(M,g)$ is a Riemannian submersion if and only if it is a Finsler isometric submersion. The great difference is that in the Riemannian case the horizontal cone in $T_{x'}M$ is the orthogonal of $\ker(T_{x'}f)$ for any $x'\in M'$.\\
\end{Rem}

Now in the context of Finsler isometric submersion (not necessary  Finsler isometric submersion between manifolds) we have the following results on the entropy:
\begin{The}\label{subentropy}${}$\\
Let $f:M'\ap M$ be a Finsler isometric submersion between compact  Finsler manifolds $(M',\Phi')$ and $(M,\Phi)$. Consider a filtration $\{{\cal A}_r\}_{r\in \R^+}$ of a set ${\cal A}$ of admissible curve in $M$. Let 
${\cal A}'_r$  be the set of horizontal lift of curves in ${\cal A}_r$ and we set ${\cal A}'=\dis\bigcup_{r\in \R^+} {\cal A}'_r$. Then ${\cal A}'$ is a set of admissible curves in $M$ and  $\{{\cal A}'_r\}_{r\in \R^+}$ is a filtration of ${\cal A}'$. Moreover   we have  $h({\cal A}',M', d_{\Phi'})\geq  h({\cal A},M,d_\Phi)$. In particular if $f$ is an isometry  then $h({\cal A}',M', d_{\Phi'})=h({\cal A},M, d_\Phi)$.\\
\end{The}

\begin{The}\label{pullbackentropy}${}$\\
Let $f:M'\ap M$ be a   surjective map between compact  Finsler manifolds $(M',\Phi')$ and $(M,\Phi)$. Consider is a smooth singular  distribution $D$  on $M$ provided with a (weak) Finsler metric $F$  and  the  pull back distribution $D'=f^{-1}(D)$ on $M'$   provided with a (weak)  Finsler $F'$. Assume that  $f$ is transverse to $D$ and  $T_{x'}f:D_{x'}\ap D_{f(x')}$ is a Finsler submersion between these two Minskowski normed spaces. Then
$$  h(D',M', F')\geq h(D,M, F) $$
\end{The}

\begin{proof}[Proof of Theorem \ref{subentropy}]${}$\\
In the one hand, by construction for any ${\g}'\in {\cal A}'$,  $f\circ \g'$ belongs to $\cal A$, so  it follows that ${\cal A}'$ is an admissible set of curves in $M'$.  On the other hand, since  the length of a horizontal lift ${\g}'$ of some curve $\g$ is equal to the length of $\g$, we have ${\cal A}'_r=\{{\g}'\in {\cal A}'\;:\; f\circ \g\in {\cal A}_r\}$ and so  $\{{\cal A}'_r\}_{r\in \R^+}$ is a filtration of $\cal A$. Since $f$ is surjective if $x=f(x')$ we have
$$f({\cal A}'_r(x')):=\{f\circ {\g}'\;:\; {\g}'\in {\cal A}'_r(x')\}={\cal A}_r(x)$$
Moreover for any $\g'\in {\cal A}'_r(x')$ then we have  $l({\g}')=l(f\circ {\g}')$.
Since  $d_{\Phi}(f(x'),f(y'))\leq d_{\Phi'}(x',y')$ we must have $\bar{d}_{\Phi}(f\circ {\g}', f\circ \mu')\leq \bar{d}_{\Phi'}(\g',\mu')$ (notations of Remark \ref{interpretationdelta}). It follows that for any $\g'\in {\cal A}'_r(x')$ we have:
$$\bar{d}_\Phi(f\circ \g',{\cal A}_r(f(y')))\leq \bar{d}_{\Phi'}(\g',{\cal A}'_r(y')).$$  Therefore for any $x$ and $y$ in $M$ and any $x'\in f^{-1}(x)$
 and $y'\in f^{-1}(y)$  we get:
$$\dis\sup_{\g\in {\cal A}_r(x)}  \bar{d}_\Phi(\g,{\cal A}_r(y))\leq \dis\sup_{\g'\in {\cal A}'_r(x')} \bar{d}_\Phi(f\circ \g',{\cal A}_r(y))\leq \dis\sup_{\g'\in {\cal A}'_r(x')}\bar{d}_{\Phi'}(\g',{\cal A}'_r(y')).$$
This implies that,  if $\{d_r\}_{r\in\R^+}$ and $\{d'_r\}_{r\in\R^+}$ are the family of distances associated to the filtration $\{{\cal A}_r\}_{r\in \R^+}$ and $\{{\cal A}'_r\}_{r\in \R^+}$ respectively, we have
$$d_r(f(x'),f(y'))\leq d'_r(x',y')\;\; \forall x', y'\in M'.$$ 
From Proposition \ref{hproduct} Point (2) we obtain  the inequality $h({\cal A}',M', d_{\Phi'})\geq h({\cal A},M, d_\Phi)$.

\end{proof}

\begin{proof}[ Proof of Theorem \ref{pullbackentropy}]${}$\\
Denote by ${\cal A^D}'$ and ${\cal A^D}$ the set of $D'$-curves and  $D$-curves,  $\{{\cal A^D}'_r\}_{r\in \R^+}$ and $\{{\cal A^D}_r\}_{r\in \R^+}$ the associated natural filtration by length relative to $F'$ and $F$ respectively. Fom our assumption, in ${\cal A^D}'$ if $\g'$ belongs to ${\cal A^D}'$ then $\g=f\circ \g'$ belongs to ${\cal A^D}$.  Such a curve $\g'$ will be called {\it lift } of $\g$. Moreover $\g'$ will be called an {\it horizontal lift} if $\dot{\g}'$ belongs to the horizontal cone of $D'_{\g'(t)}$ {\it a.e.}.

 \noindent Since $f$ is surjective and $D'=f^{-1}(D)$, by using similar arguments to those used in the proof of Lemma \ref{Existenceconicalcontrol}, we can show that  if $\g$ belongs to ${\cal A^D}$   there always exists a  horizontal lift $\g'\in {\cal A^D}'$ of $\g$.  It follows in particular that $${\cal A^D}=\{f\circ\g'\;:\; \g'\in {\cal A^D}'\}:=f({\cal A^D}').$$
  Moreover if ${\cal H^D}'$ is the set of horizontal lifts of curves of ${\cal A^D}$ in ${\cal A^D}'$ then we also have
  $${\cal A^D}=\{f\circ\g'\;:\; \g'\in {\cal H^D}'\}:=f({\cal H^D}').$$
Now, if $\g'\in {\cal A^D}'$ is a lift of $\g\in{\cal A^D}$ then the length $l(\g')$ (relative to $F'$) is less or equal to the length $l(\g)$ (relative to $F$) and $\g'$ is an horizontal lift if and only if $l(\g')=l(\g)$. This implies that the natural filtrations $\{{\cal A^D}'_r\}_{r\in \R^+}$, $\{{\cal H^D}'_r\}_{r\in \R^+}$ and $\{{\cal A^D}'_r\}_{r\in \R^+}$ of the set  ${\cal A^D}'$, ${\cal H^D}'$ and ${\cal A^D}$ respectively satisfies the following relations:

${\cal H^D}'_r \subset {\cal A^D}'_r$ and $f({\cal H^D}'_r)={\cal A^D}_r $ which implies $f({\cal A^D}'_r)= {\cal A^D}_r$ since $l(\g')\leq  l(f\circ\g')$.

\noindent Denote by $SM'$ the set of vector $v'\in T_{x'}M'$ such that $\Phi'(x',v')=1$. Since $SM'$ is compact and since the map $(x',v') \mapsto \Phi'(f(x'),Tf(v'))$ is continuous on $SM'$  this map is bounded  on $SM'$ by some constant $K$. Therefore we have 
$d_{\Phi}(f(x'),f(y'))\leq K d_{\Phi'}(x',y').$
By similar arguments used in the previous proof we also have
$$\dis\sup_{\g\in {\cal A^D}_r(x)}  \bar{d}_\Phi(\g,{\cal A}_r(y))\leq \dis\sup_{\g'\in {\cal A^D}'_r(x')} \bar{d}_\Phi(f\circ \g',{\cal A}_r(y))\leq \dis\sup_{\g'\in {\cal A}'_r(x')}K\bar{d}_{\Phi'}(\g',{\cal A}'_r(y')),$$
for any $x$ and $y$ in $M$ and any $x'\in f^{-1}(x)$
 and $y'\in f^{-1}(y)$
and we conclude as in the previous proof.

\end{proof}

\subsection{Singular Riemannian foliation}\label{srf}${}$\\
In this section we recall the principal results about singular Riemannian developed in \cite{Al} and  which will be used in the following subsection.\\

Let $(M,g)$ be a complete Riemannian manifold. A {\it singular Riemannian foliation} $\cal F$ on $M$ is a Stefan-Sussmann  foliation such that:


every geodesic which is perpendicular to one leaf must be perpendicular to any leaf it meets.\\

When the foliation $\cal F$ is regular we simply say that $\cal F$ is a Riemannian manifold. Such a notion was introduced by P. Molino (\cite{Mo}).
Typical examples of singular Riemannian foliations  are the partition by orbits of an isometric action or  by leaf closures of a Riemannian foliation. For more Examples see \cite{Al} and references inside this paper. 

Given a singular Riemannian foliation $\cal F$ on $M$, the union of the leaves having the same dimension is an embedded submanifold called stratum and in particular the {\it minimal stratum}
 is a closed submanifold (see \cite{Mo}). In fact, each stratum is an embedded submanifold and a union of
geodesics that are perpendicular to the leaves. \\The essential result of \cite{Al} is the following desingularization Theorem:
 
 \begin{The}\label{desing}${}$\\
 Let $\cal F$ be a singular Riemannian foliation of a compact Riemannian manifold $(M,g)$, $\S$ the minimal stratum of $\cal F$ ( with leaves of dimension $k_0$) and $\rm{Tub}_r(\S)$ the tubular neighborhood over $\S$ of radius $r$ (relative to $g$). Then, by blowing up $M$ along $\S$, we obtain a singular Riemannian foliation $\hat{\cal F}_r$ (with leaves of dimension greater then $k_0$) on a compact Riemannian manifold $(\hat{M}_r(\S),\hat{g}_r)$, and a map 
 $\hat{\pi}_r: \hat{M}_r(\S) \ap M$ with the following properties:

 (a) $\hat{\pi}_r$ projects each leaf of $\hat{\cal F}_r$ into a leaf of ${\cal F}$.

 (b) Set $\hat{\S}=\hat{\pi}_r^{-1}(\S)$.Then $\hat{\pi}_r :(\hat{M}_r(\S)\setminus \hat{\S},\hat{\cal F})\ap (M\setminus \S,{\cal F})$ is a foliated
diffeomorphism and $\hat{\pi}_r :(\hat{M}_r(\S)\setminus\rm{Tub}_r (\hat{\S}),\hat{\cal F})\ap (M\setminus \rm{Tub}_r(\S),{\cal F})$ is an isometry.
 
 (c) If a unit speed geodesic $\hat{\g}$ is orthogonal to $\hat{\S}$, then $\hat{\pi}_r(\hat{\g})$ is a unit speed geodesic orthogonal to $\S$.
 
 (d) ${(\hat{\pi}_r)}_{| \hat{\S}} :(\hat{\S},\hat{g}_r)\ap (\S,g)$ is a Riemannian submersion. In addition $(\hat{\S},\hat{\cal F}_{|\hat{\S}},\hat{g}_r)$ is a singular \\ ${}\;\;\;$Riemannian foliation.
  
  \noindent Moreover the liftings of horizontal geodesics of $(\S,{\cal F}_{| \S},g)$  are horizontal geodesics of $(\hat{\S},\hat{\cal F}_{|\hat{\S}},\hat{g}_r)$.
  
\noindent Furthermore, by successive blow-ups, we have a regular Riemannian foliation $\hat{F}$ on a compact Riemannian manifold $\hat{M}$ and a desingularization map $\hat{\rho} : \hat{M} \ap  M$ that projects each leaf $\hat{L}$ of $\hat{\cal F}$ into a leaf $L$ of $\cal F$.
  \end{The}
  We end this section by a Corollary  of  this theorem which be used in the next section. 
 \begin{Cor}\label{propF}${}$\\
 With the same assumptions and notations of Theorem \ref{desing} we have the following properties
  \begin{enumerate}
  \item {\it the map  $\hat{\pi}_r: \hat{M}_r(\S) \ap M$ is transverse to the foliation $\cal F$} ({\it i.e.} to the associated distribution $D$ of $\cal F$. 
  \item The foliation $\hat{\cal F}_r$ of $\hat{M}_r(\S) $ is the pull-back of the foliation $\cal F$ of $M$
  \item There exists a Riemannian metric $\hat{g}_M$ on $M$ with the following property:
   if $\hat{L}$ is a leaf of $\hat{\cal F}_r$  which projects on a leaf $L$ of $\cal F$ (via $\pi_r$) then $T\pi_r:T_{\hat{x}}\hat{L}\ap T_{f(\hat{x}}L$ is a Finsler isometric submersion with respect to the Minkowski norm on $T_{\hat{x}}\hat{L}$ 
  and $T_{f(\hat{x}}L$ induced by $\hat{g}_r$ and $\hat{g}_M$ respectively.
  \end{enumerate}
  \end{Cor}
  
  \begin{proof} ${}$\\
\noindent  {\it Point (1)}: Let  $\hat{x}$  be a point in $\hat{M}_r(\S)$. Then either $\hat{x}$ belongs to $\hat{M}_r(\S)\setminus \hat{\S}$ and the result comes from Point (b)
or $\hat{x}$ belongs to $ \hat{\S}$ and the result comes from point (d).\\

\noindent  {\it Point (2)}: Again it is a consequence of Point (b) and (d) of Theorem \ref{desing}.\\
 
  \noindent {\it Point (3)}: From Point (b) of Theorem \ref{desing}, via the  foliated diffeomorphism  
  $$\hat{\pi}_r :(\hat{M}_r(\S)\setminus \hat{\S},\hat{\cal F})\ap (M\setminus \S,{\cal F})$$ 
   from $\hat{g}_r$ we obtain a Riemannian metric $\hat{g}_M$ on $M\setminus \S$ such that the restriction of $\hat{\pi}_r$ to a leaf $\hat{L}$ of $\hat{\cal F}_{| \hat{M}_r(\S)\setminus \hat{\S}}$ on a leaf  $L$ of ${\cal F}_{| M\setminus \S}$ is an isometry. Note that $\hat{g}_M=g$ on $M\setminus \rm{Tub}_r(\S)$ from Point (b) and $\hat{g}_M$ is exactly the Riemannian metric on  $\rm{Tub}_r(\S)\setminus \S$ build in \cite{AnSk} in Proposition 3.2. Now the proof of Lemma 3.5 in \cite{AnSk} implies that $\hat{g}_M$ can be smoothly prolonged on $\S$ by the initial metric $g$. Now the result for the leaves of ${\cal F}_{| \hat{\S}}$ comes from   Point (d) of Theorem \ref{desing}.

  \end{proof}

  \subsection{Finsler entropy of singular Riemannian foliation (Proof of Theorem \ref{The2})}${}$\\
  Let $\cal F$ be a singular Riemannian foliation $\cal F$ on a compact Riemannian foliation $(M,g)$.   We equip the smooth distribution $D$ defined by $\cal F$ of the Riemannian metric $g_D$ induced by $g$. Then we can consider the  geometric entropy  $h({\cal F},M,g)$ of $\cal F$ defined in Definition \ref{entropycalF}. We will show that  $h({\cal F},M,g)=0$   
   \begin{proof}[Proof of Theorem \ref{The2}]${}$\\
   At first from Observation \ref{propentrop} Point {\bf 2}, without loss of generality we may choose any Riemannian metric on $M$ to prove the result.  
   From Corollary \ref{propF}, and Theorem \ref{pullbackentropy} we have
 \begin{eqnarray}\label{hblowup}
h(\hat{\cal F}, \hat{M}_r,\{\hat{g}_r\}_{|\hat{\cal F}_r})\geq h({\cal F},M,\{ {g}_M\}_{| {\cal F}})\geq 0
\end{eqnarray}

   where $\{\hat{g}_r\}_{|\hat{\cal F}_r}$ and $\{ \hat{g}_M\}_{| {\cal F}}$ are the induced Riemannian metric on the distributions generated by $\hat{\cal F}_r$ and ${\cal F}$ respectively.
   Now  each singular Riemannian foliation of a compact manifold has a stratification   $\{\S_k\}_{k=1,\dots,d}$, such that each $\S_k$ is the union of leaves of same dimension ({\it cf.} section \ref{srf} or more precisely see \cite{Mo}). 
   It is well known that the geometric entropy of $\cal F$ is zero (see \cite{GLW} or \cite{Wa}). Since the geometric entropy is nothing but else $h({\cal F},M,g_{\cal F})$ as defined in Definition \ref{entropycalF}, the result is true if $\cal F$ is regular
  Assume that the smallest  dimension of a leaf is  $m$. Then, if $\cal F$ is singular, the bowing-up $\pi_r:\hat{M}_r\ap M$ produces a singular Riemannian foliation ${\cal F}_r$ whose smallest dimension  of a leaf is $m_r>m$.  Therefore after a finite sequence of blow-up we obtain a regular Riemannian foliation on a compact Riemannian manifold. According the the relation between the  Finsler entropy of a singular Riemannian foliation and the Finsler entropy of the singular Riemannian foliation obtained by blowing up, according to Observation \ref{propentrop} Point {\bf 2}, it follows that that $h({\cal F},M,\{ {g}_M\}_{| {\cal F}})=0$

   \end{proof}

\end{document}